\documentclass{amsart}

\usepackage{amssymb,graphicx,tikz}
\usepackage{hyperref}

\usepackage{todonotes}
\usepackage{youngtab}
\usepackage{xypic}
\usepackage{graphicx}
\usepackage{tikz}

\newtheorem{theorem}{Theorem}[section]

\newtheorem{corollary}[theorem]{Corollary}
\newtheorem{prop}[theorem]{Proposition}

\theoremstyle{definition}
\newtheorem{definition}[theorem]{Definition}
\newtheorem{example}[theorem]{Example}
\newtheorem{problem}[theorem]{Problem}
\newtheorem{xca}[theorem]{Exercise}

\theoremstyle{remark}
\newtheorem{remark}[theorem]{Remark}
\newtheorem*{hint}{Hint}

\numberwithin{equation}{section}

\newcommand{\CC}{{\mathbb{C}}}
\newcommand{\KK}{{\mathbb{K}}}
\newcommand{\PP}{{\mathbb{P}}}
\newcommand{\ZZ}{{\mathbb{Z}}}
\newcommand{\FF}{{\mathbb{F}}}
\newcommand{\RR}{{\mathbb{R}}}

\newcommand{\ccL}{{\mathcal{L}}}

\newcommand{\cV}{{\mathcal{V}}}

\newcommand{\Fl}{{\mathcal{F}l}}

\newcommand{\fS}{{\mathfrak{S}}}

\newcommand{\qbinom}{\genfrac[]{0pt}{}}

\DeclareMathOperator{\GL}{GL}
\DeclareMathOperator{\Gr}{Gr}
\DeclareMathOperator{\Stab}{Stab}
\DeclareMathOperator{\rk}{rk}
\DeclareMathOperator{\vol}{vol}
\DeclareMathOperator{\Diff}{Diff}
\DeclareMathOperator{\Ann}{Ann}
\DeclareMathOperator{\Ker}{Ker}
\DeclareMathOperator{\codim}{codim}
\DeclareMathOperator{\mult}{mult}
\DeclareMathOperator{\ch}{ch}
\DeclareMathOperator{\wt}{Wt}

\newenvironment{rcgraph}{\begin{trivlist}\item\centering\footnotesize$}
			{$\end{trivlist}}

\def\textcross{\ \smash{\lower2pt\hbox{\rlap{\hskip2.07pt\vrule height10pt}}
                \raise2.8pt\hbox{\rlap{\hskip-3pt \vrule height.4pt depth0pt
                width10.7pt}}}\hskip10.7pt\!\!}
\def\textelbow{\ \hskip.1pt\smash{\raise2.8pt%
                \hbox{\co \hskip 2.07pt\rlap{\rlap{\char'005} \char'007}
                \lower5.2pt\rlap{\vrule height1.5pt}
                \raise4.2pt\rlap{\vrule height1.5pt}}
                \raise2.8pt\hbox{%
                  \rlap{\hskip-5.25pt \vrule height.4pt depth0pt width1.5pt}%
                  \rlap{\hskip4.25pt \vrule height.4pt depth0pt width1.5pt}}}
                \hskip6.7pt}

\font\co=lcircle10
\def\petit#1{{\scriptstyle #1}}

\def\jr{\smash{	\raise2pt\hbox{\co \rlap{\rlap{\char'005} \char'007}}
		\raise6pt\hbox{\rlap{\vrule height2pt}}
		\raise2pt\hbox{\rlap{\hskip4pt \vrule height0.4pt depth0pt
                 width2.5pt}}
		\raise2pt\hbox{\rlap{\hskip-6pt \vrule height.4pt depth0pt
                 width2.2pt}}
		\lower4pt\hbox{\rlap{\vrule height2.5pt}}}}
\def\je{\smash{\raise2pt\hbox{\co \rlap{\rlap{\char'005}
                \phantom{\char'007}}}\raise6pt\hbox{\rlap{\vrule height2pt}}
	       \raise2pt\hbox{\rlap{\hskip-6pt \vrule height.4pt depth0pt
                width2.2pt}}}}
\def\+{\smash{\lower4pt\hbox{\rlap{\vrule height12pt}}
                \raise2pt\hbox{\rlap{\hskip-6pt \vrule height.4pt depth0pt
                width12.5pt}}}}
\def\er{\smash{	\raise2pt\hbox{\co \rlap{\phantom{\rlap{\char'005}} \char'007}}
		\raise6pt\hbox{\rlap{\phantom{\vrule height2pt}}}
		\raise2pt\hbox{\rlap{\hskip4pt \vrule height0.4pt depth0pt
                 width2.5pt}}
		\raise2pt\hbox{\rlap{\phantom{%
		 \hskip-6pt \vrule height.4pt depth0pt width2.2pt}}}
		\lower4pt\hbox{\rlap{\vrule height2.5pt}}}}
\def\hor{\smash{\lower2pt\hbox{\rlap{\phantom{\vrule height10pt}}}
                \raise2pt\hbox{\rlap{\hskip-6pt \vrule height.4pt depth0pt
                width12.5pt}}}}
\def\ver{\smash{\lower2pt\hbox{\rlap{\vrule height10pt}}
                \raise2pt\hbox{\rlap{\phantom{%
		\hskip-6pt \vrule height.4pt depth0pt width12.5pt}}}}}

\def\perm#1#2{\hbox{\rlap{$\petit {#1}_{\scriptscriptstyle #2}$}}%
                \phantom{\petit 1}}

	{%
	 \def\*{\makebox[0ex]{\footnotesize$+\,$}}%
	 \begin{array}{|*{#1}{@{\ \;\:}c|}}}
	{\end{array}}

	{%
	 \begin{array}{c*{#1}{@{\ \ \;}c}}}
	{\end{array}}

\newenvironment{shortstretchpipedream}[1]%
	{%
	 \begin{array}{c*{#1}{@{\ \ \;}c}}}
	{\end{array}}

	{%
	 \begin{array}{c*{#1}{@{\ \ \;}c}}}
	{\end{array}}

\begin{document}

\title{Grassmannians, flag varieties, and Gelfand--Zetlin polytopes}

\author{Evgeny Smirnov}
\address{Faculty of Mathematics and Laboratory of Algebraic Geometry and its Applications, National Research University Higher School of Economics, Vavilova 7, Moscow 119312, Russia}
\address{Laboratoire franco-russe J.-V.~Poncelet (UMI~2615 du CNRS), Bolshoi Vlassievskii per., 11, Moscow 119002, Russia}
\email{esmirnov@hse.ru}
\thanks{Partially supported by Russian Science Foundation, RScF project 14-11-00414, Dynasty Foundation fellowship and Simons--IUM fellowship.}

\date{\today}


\dedicatory{To the memory of Andrei Zelevinsky}

\begin{abstract} These are extended notes of my talk given at Maurice Auslander Distinguished Lectures and International Conference (Woods Hole, MA) in April 2013. Their aim is to give an introduction into Schubert calculus on Grassmannians and flag varieties. We discuss various aspects of Schubert calculus, such as applications to enumerative geometry, structure of the cohomology rings of Grassmannians and flag varieties, Schur and Schubert polynomials. We conclude with a survey of results of V.~Kiritchenko, V.~Timorin and the author on a new approach to Schubert calculus on full flag varieties via combinatorics of Gelfand--Zetlin polytopes.
\end{abstract}

\maketitle

\section{Introduction}

\subsection{Enumerative geometry}

Enumerative geometry deals with problems about finding the number of geometric objects satisfying certain conditions. The earliest problem of that kind was probably formulated (and solved) by Apollonius of Perga around 200 BCE:

\begin{problem}[Apollonius] Find the number of circles in the plane which are tangent to three given circles.
\end{problem}

Of course, the answer depends on the mutual position of the three given circles. For instance, if all circles are contained inside each other, no other circle can be tangent to all three. It turns out that for any number not exceeding 8 and not equal to 7 there exists a configuration of three circles such that the number of circles tangent to all of them is equal to this number. All these circles can be explicitly constructed with compass and straightedge.
 
Starting from the early 19th century mathematicians started to consider enumerative problems in projective geometry. The development of projective geometry is usually associated with the name of a French mathematician and military engineer Jean-Victor Poncelet. In his work ``Trait\'e des propri\'et\'es projectives des figures'', written during his imprisonment in Russia after Napoleon's campaign in 1813--1814 and published in 1822, Poncelet made two important choices: to work over the complex numbers rather than over the real numbers, and to work in projective space rather than the affine space. For example, with these choices made, we can say that a conic and a line in the plane always intersect in two points (counted with multiplicity), while for a conic and a line in the real affine plane the answer can be 0, 1, or 2. This is the first illustration of Poncelet's ``continuity principle'', discussed below.

In terms of complex projective geometry, a circle on the real plane is a smooth conic passing through two points $(1:i:0)$ and $(1:-i:0)$ at infinity. So the problem of Apollonius is essentially about the number of conics passing through given points and tangent to given conics. In 1848 Jacob Steiner dropped the condition that all conics pass through two given points and asked how many conics on the plane are tangent to given five conics. He also provided an answer to this problem: he claimed that this number is equal to $7776=6^5$. This number is so large that it cannot be checked by construction. However, this answer turned out to be wrong. Steiner did not give a complete solution to this problem; he just observed that the number of conics tangent to a given conic and passing through four given points is equal to 6, the number of conics tangent to two given conics and passing through three points is $36=6^2$, and so on. This fails already on the next step: the number $6^3$ gives an upper bound for the number of conics tangent to two conics and passing through  three points, but the actual number of such curves is always less than that!

In 1864 Michel Chasles published a correct answer\footnote{Sometimes this result is attributed to Ernest de Jonqui\`eres, a French mathematician, naval officer and a student of Chasles, who  never published it.} to Steiner's problem: the number of conics tangent to given five is equal to 3264.  Chasles found out that the number of conics in a one-parameter family that satisfy a single condition can be expressed in the form $\alpha\mu+\beta\nu$, where $\alpha$ and $\beta$ depend only on the condition (they were called \emph{characteristics}), while $\mu$ and $\nu$ depend only on the family: $\mu$ is the number of conics in the family passing through a given point and $\nu$ is the number of conics in the family tangent to a given line.

Given five conditions with ``characteristics'' $\alpha_i$ and $\beta_i$, Chasles found an expression for the number of conics satisfying all five. In 1873, Georges Halphen observed that Chasles's expression factors formally into the product
\[
 (\alpha_1\mu+\beta_1\nu)(\alpha_2\mu+\beta_2\nu)\dots(\alpha_5\mu+\beta_5\nu),
\]
provided that, when the product is expanded, $\mu^i\nu^{5-i}$ is replaced by the number of conics passing through $i$ points and tangent to $5-i$ lines.

This example inspired a German mathematician Hermann Schubert to develop a method for solving problems of enumerative geometry, which he called \emph{calculus of conditions}, and which is now usually referred to as \emph{Schubert calculus}. It was used to solve problems involving objects defined by \emph{algebraic} equations, for example, conics or lines in 3-space. Given certain such geometric objects,  Schubert represented conditions on them by algebraic symbols. 
Given two conditions, denoted by $x$ and $y$, he represented the new condition of imposing one or the other by $x+y$ and the new condition of imposing both simultaneously by $xy$. The conditions $x$ and $y$ were considered equal if they represented conditions equivalent for enumerative purposes, that is, if the number of figures satisfied by the conditions $xw$ and $yw$ were equal for every $w$ representing a condition such that both numbers were finite. Thus the conditions were formed into a ring.

For example, Chasles's expression $\alpha\mu+\beta\nu$ can be interpreted as saying that a condition on conics with characteristics $\alpha$ and $\beta$ is equivalent to the condition that the conic pass through any of $\alpha$ points or tangent to any of $\beta$ lines, because the same number of conics satisfy either condition  and simultaneously the condition to belong to any general one-parameter family. Furthermore, we can interpret Halphen's factorization as taking place in the ring  of conditions on conics.

One of the key ideas used by Schubert was as follows: two conditions are equivalent if one can be turned into the other by continuously varying the parameters on the first condition. This idea goes back to Poncelet, who called it \emph{the principle of continuity}, and said it was considered an axiom by many. However, it was criticized by Cauchy and others. Schubert called it first \emph{principle of special position} and then \emph{principle of conservation of number}.

For example, the condition on conics to be tangent to a given smooth conic is equivalent to the condition to be tangent to any smooth conic, because the first conic can be continuously translated to the second. Moreover, a smooth conic can be degenerated in a family into a pair of lines meeting at a point. Then the original condition is equivalent to the condition to be tangent to either line or to pass through the point. However, the latter condition must be doubled, because in a general one-parameter family  of conics, each conic through the point is the limit of two conics tangent to a conic in the family. Thus the characteristics on the original condition are $\alpha=\beta=2$.

As another example, let us consider the famous problem about four lines in  3-space, also dating back to Schubert. In this paper we will use this problem as a baby example to demonstrate various methods of Schubert calculus (see Example~\ref{ex:schubert_pieri} or the discussion at the end of Subsection~\ref{ssec:decomp} below).

\begin{problem}\label{prob:schubert} Let $\ell_1,\ell_2,\ell_3,\ell_4$ be four generic lines in a three-dimensional complex projective space. Find the number of lines meeting all of them.
\end{problem}

The solution proposed by Schubert was as follows. The condition on a line $\ell$ in 3-space to meet two skew lines $\ell_1$ and $\ell_2$ is equivalent to the condition that $\ell$ meet two intersecting lines. The same can be said about the lines $\ell_3$ and $\ell_4$. So the initial configuration can be degenerated in such a way that the first two lines would span a plane and the second two lines would span another plane. The number of lines intersecting all four would then remain the same according to the principle of conservation of the number. And for such a degenerate configuration of lines $\ell_1,\dots,\ell_4$ it is obvious that there are exactly two lines intersecting all of them: the first one passes through the points $\ell_1\cap\ell_2$ and $\ell_3\cap \ell_4$, and the other is obtained as the intersection of the plane spanned by $\ell_1$ and $\ell_2$ with the plane spanned by $\ell_3$ and $\ell_4$.

In his book ``Kalk\"ul der abz\"ahlenden Geometrie'' \cite{Schubert79}, published in 1879, Schubert proposed what he called the \emph{characteristic problem}. Given figures of fixed sort and given an integer $i$, the problem is to find a basis for the $i$-fold conditions (i.e., the conditions restricting freedom by $i$ parameters) and to find a dual basis for the $i$-parameter families, so that every $i$-fold condition is a linear combination of basis $i$-fold conditions, and so that the combining coefficients, called the ``characteristics'', are rational numbers, which can be found as the numbers of figures in the basic families satisfying the given conditions. We have already seen this approach in the example with the conics tangent to given five.

In his book Schubert solved the characteristics problem for a number of cases, including conics in a plane, lines in 3-space, and point-line flags in 3-space. In some other cases, he had a good understanding of what these basis conditions should be, which allowed him to find the number of figures satisfying various combinations of these conditions. In particular, he computed the number of twisted cubics
tangent to 9 general quadric surfaces in 3-space and got the right answer: 5,819,539,783,680; a really impressive achievement for the pre-computer era!

In 1886 Schubert solved the general case of characteristic problem for projective subspaces. For this he introduced the Schubert cycles on the Grassmannian and, in modern terms, showed that they form a self-dual basis of its cohomology group. Further, he proved the first case of the Pieri rule, which allowed him to compute the intersection of a Schubert variety with a Schubert divisor. Using this result, he showed that the number of $k$-planes in an $n$-dimensional space meeting $h$ general $(n-d)$-planes is equal to $\frac{1!\cdot 2!\cdot\dots \cdot k! \cdot h!}{(n-k)!\cdot\dots \cdot n!}$, where $h=(k+1)(n-k)$. In other words, he found the degree of the Grassmannian $\Gr(k,n)$ under the Pl\"ucker embedding. We will discuss these results in Section~\ref{sec:grassmann}.

With this new technique Schubert solved many problems which had already been solved, and many other problems which previously defied solution. Although his methods, based on the principle of conservation of number, lacked rigorous foundation, there was no doubt about their validity. In 1900, Hilbert formulated his famous list of 23 problems. The 15th problem was entitled ``Rigorous foundation of Schubert's enumerative calculus'', but in his discussion of the problem he made clear that he wanted Schubert's numbers to be checked.

In the works of Severi, van der Waerden and others, Schubert calculus was given a rigorous reinterpretation. To begin with, we need to define a variety parametrizing all the figures of the given sort. An $i$-parameter family corresponds to an $i$-dimensional subvariety, while an $i$-fold condition yields a cycle of codimension $i$, that is, a linear combination of subvarieties of codimension $i$. The sum and product of conditions becomes the sum and intersection product of cycles.

The next step is to describe the ring of conditions. For this van der Waerden proposed to use the topological intersection theory. Namely, each cycle yields a cohomology class in a way preserving sum and product. Moreover, continuously varying the parameters of a condition, and so the cycle, does not alter its class; this provides us with a rigorous interpretation of the principle of conservation of number. 

Furthermore, the cohomology groups are finitely generated. So we may choose finitely many basic conditions and express the class of any condition uniquely as a linear combination of those. Thus an important part of the problem is to describe the algebraic structure of the cohomology ring of the variety of all figures of the given sort. We will provide (to some extent) such a description for Grassmannians, i.e., varieties of $k$-planes in an $n$-space, and full flag varieties.

Finally, it remains to establish the enumerative significance of the numbers obtained in computations with the cohomology ring. For this we need to consider the action of the general linear group on the parameter variety for figures and to ask whether the intersection of one subvariety and a general translate of the other is transversal. Kleiman's transversality theorem, which we discuss in Subsection~\ref{ssec:kleiman}, asserts that the answer is affirmative if the group acts transitively on the parameter space; in particular, this is the case for Grassmannians and flag varieties.

\subsection{Structure of this paper} The main goal of this paper is to give an introduction into Schubert calculus. More specifically, we will speak about Grassmannians and complete flag varieties. We restrict ourselves with the type $A$, i.e., homogeneous spaces of $\GL(n)$.

In Section~\ref{sec:grassmann} we define Grassmannians, show that they are projective algebraic varieties and define their particularly nice cellular decomposition: the Schubert decomposition. We show that the cells of this decomposition are indexed by Young diagrams, and the inclusion between their closures, Schubert varieties, is also easily described in this language. Then we pass to the cohomology rings of Grassmannians and state the Pieri rule, which allows us to multiply cycles in the cohomology ring of a Grassmannian by a cycle of some special form. Finally, we discuss the relation between Schubert calculus on Grassmannians and theory of symmetric functions, in particular, Schur polynomials.

Section~\ref{sec:flags} is devoted to full flag varieties. We mostly follow the same pattern: define their Schubert decomposition, describe the inclusion order on the closures of Schubert cells, describe the structure of the coholomology ring of a full flag variety and formulate the Monk rule for multiplying a Schubert cycle by a divisor. Then we define analogues of Schur polynomials, the so-called Schubert polynomials, and discuss the related combinatorics. The results of these two sections are by no means new, they can be found in many sources; our goal was to present a short introduction into the subject. A more detailed exposition can be found, for example, in \cite{Fulton97} or \cite{Manivel98}.

In the last two sections we discuss a new approach to Schubert calculus of full flag varieties, developed in our recent joint paper \cite{KiritchenkoSmirnovTimorin12} with Valentina Kiritchenko and Vladlen Timorin. This approach 
uses some ideas and methods from the theory of toric varieties (despite the fact that flag varieties are not toric). In Section~\ref{sec:toric} we recall some notions related with toric varieties, including the notion of the Khovanskii--Pukhlikov ring of a polytope. Finally, in Section~\ref{sec:gz} we state our main results: to each Schubert cycle we assign a linear combination of faces of a Gelfand--Zetlin polytope (modulo some relations) in a way respecting multiplication: the product of Schubert cycles corresponds to the intersection of the sets of faces. Moreover, this set of faces allows us to find certain invariants of the corresponding Schubert variety, such as its degree under various embedding.

This text is intended to be introductory, so we tried to keep the exposition  elementary and to focus on concrete examples whenever possible. 

As a further reading on combinatorial aspects of Schubert calculus, we would recommend the books \cite{Fulton97}  by William Fulton and \cite{Manivel98} by Laurent Manivel. The reader who is more interested in geometry might want to look at the wonderful lecture notes by Michel Brion \cite{Brion05} on geometric aspects of Schubert varieties or the book \cite{BrionKumar05} by Michel Brion and Shrawan Kumar on Frobenius splitting and its applications to geometry of Schubert varieties. However, these texts are more advanced and require deeper knowledge of algebraic geometry.

More on the history of Schubert calculus and Hilbert's 15th problem can be found in Kleiman's paper on Hilbert's 15th problem \cite{Kleiman76} or in the preface to the 1979 reprint of Schubert's book \cite{Schubert79}.

\subsection*{Acknowledgements} These are extended notes of my talk given at Maurice Auslander Distinguished Lectures and International Conference (Woods Hole, MA) in April 2013. I am grateful to the organizers of this conference: Kyoshi Igusa, Alex Martsinkovsky, and Gordana Todorov, for their kind invitation. I also express my deep gratitude to the referee for their valuable comments, especially on the introductory part, which helped to improve the paper.

These notes are also based on my minicourse ``Geometry of flag varieties'' which was given in June 2012 at the Third School and Conference ``Lie Algebras, Algebraic Groups and Invariant Theory'' in Togliatti, Russia, and in September 2013 at the University of Edinburgh, Scotland, in the framework of the LMS program ``Young British and Russian Mathematicians''. They were partially written up during my visit to the University of Warwick, England, in September--October 2014. I would like to thank these institutions and personally Ivan Cheltsov and Miles Reid for their warm hospitality.

I dedicate this paper to the memory of Andrei Zelevinsky, who passed away in April 2013, several days before the Maurice Auslander Lectures. Andrei's style of research, writing and teaching mathematics will always remain a wonderful example and a great source of inspiration for me.

\section{Grassmannians}\label{sec:grassmann}

\subsection{Definition} Let $V$ be an $n$-dimensional vector space over $\CC$, and let $k<n$ be a positive integer.

\begin{definition} A \emph{Grassmannian} (or a \emph{Grassmann variety}) of $k$-planes in $V$ is the set of all $k$-dimensional vector subspaces $U\subset V$. We will denote it by $\Gr(k,V)$.
\end{definition}


\begin{example} For $k=1$, the Grassmannian $\Gr(1,V)$ is nothing but the projectivization $\PP V$ of the space $V$.
\end{example}

Our first observation is as follows: $\Gr(k,V)$ is a \emph{homogeneous $\GL(V)$-space}, i.e., the group $\GL(V)$ of nondegenerate linear transformations of $V$ acts transitively on $\Gr(k,V)$. Indeed, every $k$-plane can be taken to any other $k$-plane by a linear transform.

Let us compute the stabilizer of a point $U\in\Gr(k,V)$ under this action. To do this,  pick a basis $e_1,\dots,e_n$ of $V$ and suppose that $U$ is spanned by the first $k$ basis vectors: $U=\langle e_1,\dots,e_k\rangle$. We see that this stabilizer, which we denote by $P$,  consists of nondegenerate block matrices with zeroes on the intersection of  the first $k$ columns and the last $n-k$ columns:
\[
P=\Stab_{\GL(V)}U=\begin{pmatrix} * & *\\ 0 & *\end{pmatrix}.
\]

A well-known fact from the theory of algebraic groups states that for an algebraic group $G$ and its algebraic subgroup $H$ the set $G/H$ has a unique structure of a quasiprojective variety such that the standard $G$-action on $G/H$ is algebraic (cf., for instance, \cite[Sec.~3.1]{OnishchikVinberg90}). Since $P$ is an algebraic subgroup in $\GL(V)$, this means that $\Gr(k,V)$ is a quasiprojective variety. In the next subsection we will see that it is a projective variety.

\begin{remark} One can also work with Lie groups instead of algebraic groups. The same argument shows that $\Gr(k,V)$ is a smooth complex-analytic manifold.
\end{remark}

The dimension of $\Gr(k,V)$ as a variety (or, equivalently, as a smooth manifold) equals the dimension of the group $\GL(V)$ minus the dimension of the stabilizer of a point:
\[
\dim\Gr(k,V)=\dim\GL(V)-\dim P=k(n-k).
\]

The construction of $\Gr(k,V)$ as the quotient of an algebraic group $\GL(V)$ over its parabolic subgroup makes sense for any ground field $\KK$, not necessarily $\CC$ (and even not necessarily algebraically closed).  Note that $\GL(n,\KK)$ acts transitively on the set of $k$-planes in $\KK^n$ for an arbitrary field $\KK$, so the $\KK$-points of this variety bijectively correspond to $k$-planes in $\KK^n$.

In particular, we can consider a Grassmannian over a finite field $\FF_q$ with $q$ elements. It is an algebraic variety over a finite field; its $\FF_q$-points correspond to $k$-planes in $\FF_q^n$ passing through the origin. Of course, the number of these points is finite.

\begin{xca}\label{xca:finitefield}  Show that the number of points in $\Gr(k,\FF_q^n)$ is given by the following formula:
\[
\#\Gr(k,\FF_q^n)=\frac{(q^n-1)(q^n-q)\dots(q^n-q^{n-k+1})}{(q^k-1)(q^k-q)\dots(q^k-q^{k-1})}.
\]
This expression is called a \emph{$q$-binomial coefficient} and denoted by $\qbinom{n}{k}_q$.
Show that this expression is a polynomial in $q$ (i.e., the numerator is divisible by the denominator) and  its value for $q=1$ equals the ordinary binomial coefficient $\binom{n}{k}$.
\end{xca}

\subsection{Pl\"ucker embedding} Our next goal is to show that it is a \emph{projective variety}, i.e., it can be defined as the zero locus of a system of homogeneous polynomial equations in a projective space. To do this, let us construct an embedding of $\Gr(k,V)$ into the projectivization of the $k$-th exterior power $\Lambda^k V$ of $V$.

Let $U$ be an arbitrary $k$-plane in $V$. Pick a basis $u_1,\dots, u_k$ in $U$ and consider the exterior product of these vectors $u_1\wedge\dots \wedge u_k\in\Lambda^k V$. For any other basis $u_1',\dots,u_k'$ in $U$, the exterior product of its vectors is proportional to $u_1,\dots, u_k$, where the coefficient of proportionality equals the determinant of the corresponding base change. This means that a subspace $U$ defines an element in $\Lambda^k V$ \emph{up to a scalar}, or, in different terms, defines an element $[u_1\wedge\dots\wedge u_k]\in\PP\Lambda^k V$. This gives us a map
\[
\Gr(k,V)\to \PP\Lambda^k V.
\]
This map is called \emph{the Pl\"ucker map}, or \emph{the Plucker embedding}.

\begin{xca} Show that the Pl\"ucker map is injective: distinct $k$-planes are mapped into distinct elements of $\PP\Lambda^k V$.
\end{xca}

To show that it is indeed an embedding, we need to prove the injectivity of its differential and the existence of a polynomial inverse map in a neighborhood of each point. This will be done further, in Corollary~\ref{cor:gr_algebraic}.

An obvious but important feature of the Pl\"ucker map is that it is \emph{$\GL(V)$-equivariant}: it commutes with the natural $\GL(V)$-action on $\Gr(k,V)$ and $\PP\Lambda^k V$. In particular, its image is a closed $\GL(V)$-orbit.

A basis $e_1,\dots,e_n$ of $V$ defines a basis of $\Lambda^k V$: its elements are of the form $e_{i_1}\wedge\dots\wedge e_{i_k}$, where the sequence of indices is increasing: $1\leq i_1<i_2<\dots<i_k\leq n$. This shows, in particular, that $\dim\Lambda^k V=\binom{n}{k}$. This basis defines  a system of homogeneous coordinates on $\PP\Lambda^k V$; denote the coordinate dual to $e_{i_1}\wedge\dots e_{i_k}$ by $p_{i_1,\dots,i_k}$.

\begin{prop} The image of $\Gr(k,V)$ under the Pl\"ucker map is defined by homogeneous polynomial equations on $\PP\Lambda^k V$.
\end{prop}

\begin{proof} Recall that a multivector $\omega\in\Lambda^kV$ is called \emph{decomposable} if $\omega=v_1\wedge\dots\wedge v_k$ for some $v_1,\dots,v_k\in V$. We want to show that the set of all decomposable multivectors can be defined by polynomial equations.

Take some $\omega\in \Lambda^k V$. 
We can associate to it a map $\Phi_\omega\colon V\to\Lambda^{k+1}V$, $v\mapsto v\wedge \omega$. It is easy to see that $v\in\Ker\Phi_\omega$ iff $\omega$ is ``divisible'' by $v$, i.e., there exists a $(k-1)$-vector $\omega'\in\Lambda^{k-1}V$ such that $\omega=\omega'\wedge v$ (show this!). This means that $\dim\Ker\Phi_\omega$ equals $k$  if $\omega$ is decomposable and is less than $k$ otherwise; clearly, it cannot exceed $k$. This means that the decomposability of $\omega$ is equivalent to the inequality $\dim \Ker\Phi_\omega\geq k$. This condition is algebraic: it is given by vanishing of all its minors of order $n-k+1$ in the corresponding matrix of size $\binom{n}{k}\times n$, and these are homogeneous polynomials in the coefficients of $\omega$ of degree $n-k+1$. 
\end{proof}

\begin{example}\label{ex:g24naive} Let $n=4$ and $k=2$. The previous proposition shows that $\Gr(k,V)$ is cut out by equations of degree 3. As an exercise, the reader can try to find the number of these equations.
\end{example}

\begin{corollary}\label{cor:gr_algebraic} $\Gr(k,V)$ is an irreducible projective algebraic variety.
\end{corollary}

\begin{proof} With the previous proposition, it remains to show that $\Gr(k,V)$ is irreducible and that the differential of the Pl\"ucker map is injective at each point. The first assertion follows from the fact that $\Gr(k,V)$ is a $\GL(V)$-homogeneous variety, and $\GL(V)$ is irreducible, so $\Gr(k,V)$ is an image of an irreducible variety under a polynomial map, hence irreducible. 

Since $\Gr(k,V)$ is a homogeneous variety, for the second assertion it is enough to prove the injectivity of the differential at an arbitrary point of $\Gr(k,V)$. Let us do this for the point $U=\langle e_1,\dots,e_k\rangle$, where $e_1,\dots,e_n$ is a standard basis of $V$. Let $W\in\Gr(k,V)$ be a point from a neighborhood of $U$; we can suppose that the corresponding $k$-space is spanned by the rows of the matrix
\[
 \begin{pmatrix} 
1 & 0 & \dots & 0 & x_{11} &\dots & x_{1,n-k}\\
0 & 1 & \dots & 0 & x_{21} &\dots & x_{2,n-k}\\
\vdots & \vdots & \ddots &\vdots& \vdots & \ddots & \vdots\\
0 & 0 & \dots & 1 & x_{k1} &\dots & x_{k,n-k}\\
 \end{pmatrix}.
\]
Then all these local coordinates $x_{ij}$ can be obtained as Pl\"ucker coordinates: $x_{ij}=p_{1,2,\dots,\widehat{\imath},j+k,\dots,k}$. This means that the differential of the Pl\"ucker map is injective and locally on its image it has a polynomial inverse, so this map is an embedding. Further we will use the term ``Pl\"ucker embedding'' instead of ``Pl\"ucker map''.
\end{proof}

This ``naive'' system of equations is of a relatively high degree. In fact, a much stronger result holds.
\begin{theorem}\label{thm:plucker}
$\Gr(k,V)\subset \PP\Lambda^k V$ can be defined by a system of \emph{quadratic} equations in a \emph{scheme-theoretic} sense: there exists a system of quadratic equations generating the homogeneous ideal of $\Gr(k,V)$. These equations are called \emph{the Pl\"ucker equations}.
\end{theorem}

We will not prove this theorem here; its proof can be found, for instance, in \cite[Ch.~XIV]{HodgePedoe52}. We will only write down the Pl\"ucker equations of a Grassmannian of 2-planes $\Gr(2,n)$. For this we will use the following well-known fact from linear algebra (cf., for instance,~\cite{DummitFoote04}).

\begin{prop}\label{prop:bivect} A bivector $\omega\in \Lambda^2 V$ is decomposable iff $\omega\wedge\omega=0$. 
\end{prop}

\begin{proof}[Proof of Theorem~\ref{thm:plucker} for $k=2$]
Let 
\[
\omega=\sum_{i<j}p_{ij} e_i\wedge e_j
\]
 be a bivector. According to Proposition~\ref{prop:bivect}, it is decomposable (and hence corresponds to an element of $\Gr(k,V)$) iff 
\begin{multline*}
\omega\wedge\omega = (\sum_{i<j}p_{ij}e_i\wedge e_j)\wedge(\sum_{k<\ell}p_{k\ell}e_k\wedge e_\ell)=\\=\sum_{i<j<k<\ell}\left[p_{ij}p_{k\ell}-p_{ik}p_{j\ell}+p_{i\ell}p_{jk}\right] e_i\wedge e_j\wedge e_k\wedge e_\ell=0.
\end{multline*}
This is equivalent to
\[
p_{ij}p_{k\ell}-p_{ik}p_{j\ell}+p_{i\ell}p_{jk}=0\text{ for each }1\leq i<j<k<\ell\leq n,
\]
which gives us the desired system of quadratic equations.
\end{proof}

\begin{example}\label{ex:g24} For $\Gr(2,4)$ we obtain exactly one equation:
\[
p_{12}p_{34}-p_{13}p_{24}+p_{14}p_{23}=0.
\]
This shows that $\Gr(2,4)$ is a quadratic hypersurface in $\PP^5$. (Compare this with Example~\ref{ex:g24naive}!)
\end{example}

\subsection{Schubert cells and Schubert varieties}\label{ssec:decomp} In this subsection we construct a special cellular decomposition of $\Gr(k,V)$. The cells will be formed by $k$-planes satisfying certain conditions upon dimensions of intersection with a fixed flag of subspaces in $V$. Our exposition in the next subsections mostly follows \cite{Manivel98}.

As before, we fix a basis $e_1,\dots,e_n$ of $V$. Let $V_m$ denote the subspace generated by the first $m$ basis vectors: $V_m=\langle e_1,\dots,e_m\rangle$.

Let $\lambda$ be a partition included into the rectangle $k\times (n-k)$. This means that $\lambda$ is a nonstrictly decreasing sequence of integers: $n-k\geq \lambda_1\geq\dots\geq \lambda_k\geq 0$. Such a sequence can be associated with its \emph{Young diagram}: this is a diagram formed by $k$ rows of boxes, aligned on the left, with $\lambda_i$ boxes in the $i$-th row. Sometimes we will use the notions ``partition'' and ``Young diagram'' interchangeably. For example, here is the Young diagram corresponding to the partition $(5,4,4,1)$:
\[
 \yng(5,4,4,1)
\]
\begin{xca}
 Show that there are $\binom{n}{k}$ partitions inside the rectangle $k\times (n-k)$ (including the empty partition).
\end{xca}

To each such partition $\lambda$ we associate its \emph{Schubert cell} $\Omega_\lambda$ and \emph{Schubert variety} $X_\lambda$: these are subsets of $\Gr(k,V)$ defined by the following conditions:
\[
\Omega_\lambda=\{ U\in\Gr(k,V)\mid \dim(U\cap V_j)=i \text{ iff } n-k+i-\lambda_i\leq j\leq n-k+i-\lambda_{i+1}\}.
\]
and
\[
X_\lambda=\{U\in\Gr(k,V)\mid \dim(U\cap V_{n-k+i-\lambda_i})\geq i \text{ for }1\leq i\leq k\}.
\]

\begin{example}
 $X_\varnothing=\Gr(k,V)$ and $\Omega_{k\times(n-k)}=X_{k\times(n-k)}$ is the point $V_k$.
\end{example}
 
\begin{example} Let $\lambda(p,q)$ be the complement to a $(p\times q)$-rectangle in a $k\times (n-k)$-rectangle. Then
\[
 X_{\lambda(p,q)}=\{ U\in\Gr(k,V)\mid V_k-p\subset U\subset V_{k+q}\}\cong\Gr(p,p+q)
\]
is a smaller Grassmannian.
\end{example}

\begin{remark} Each Schubert cell $\Omega_\lambda$ contains a unique point corresponding to a subspace spanned by basis vectors, namely, $U^\lambda=\{e_{n-k+1-\lambda_1},\dots,e_{n-\lambda_k}\}$. If we consider the action of the diagonal torus $T\subset \GL(n)$ on $\Gr(k,V)$ coming from the action of $T$ on basis vectors by rescaling, then $U^\lambda$ would be a unique $T$-stable point in $\Omega_\lambda$. If $B$ is the subgroup of $\GL(V)$ which stabilizes the flag $V_\bullet$, then $\Omega_\lambda$ is the orbit of $U^\lambda$ under the action of $B$, hence a $B$-homogeneous space.
\end{remark}

\begin{prop}\label{prop:schubertdecomp} For each partition $\lambda\subset k\times (n-k)$,
\begin{enumerate}
 \item $X_\lambda$ is an algebraic subvariety of $\Gr(k,V)$, and $\Omega_\lambda$ is an open dense subset of $X_\lambda$;

\item $\Omega_\lambda\cong \CC^{k(n-k)-|\lambda|}$;

\item $X_\lambda=\overline{\Omega_\lambda}=\bigsqcup_{\mu\supset\lambda}\Omega_\mu$;

\item $X_\lambda\supset X_\mu$ iff $\lambda\subset\mu$.
\end{enumerate}
\end{prop}

\begin{proof}
 First, let us check that $X_\lambda$ is an algebraic subvariety. Indeed, the condition $\dim U\cap V_i\geq j$ can be replaced by an equivalent condition: for $U\subset V\cong \CC^n$, the rank of the map $U\to V/V_i$ is less than or equal to $n-k-j$. This is an algebraic condition, since it is given by vanishing of all minors of order $n-k-j+1$ of the corresponding matrix. The variety $X_\lambda$ is defined by such conditions, so it is algebraic.

For an arbitrary $U\in\Gr(k,V)$, the sequence of dimensions of $U\cap V_i$ goes from 0 to $k$, increasing on each step by at most one. This means that it jumps exactly in $k$ positions; we denote them by $n-k+i-\mu_i$, where $\mu$ is a partition included into the rectangle of size $k\times(n-k)$. This shows that
\[
 \Gr(k,V)=\bigsqcup_{\mu\subset k\times (n-k)} \Omega_\lambda.
\]
Moreover, if the dimension of $U\cap V_{n-k+i-\lambda_i}$ is not greater than $i$, this means that the first $i$ dimension jumps were on positions with numbers not greater than $n-k+i-\lambda_i$, which is greater than or equal to $n-l+i-\mu_i$. This means that
\[
 X_\lambda=\bigsqcup_{\mu\supset\lambda}\Omega_\mu.
\]

If $e_1,\dots,e_n$ is our standard basis of $V$ and if $U\in\Omega_{\lambda}$, this means that $U$ has a basis $u_1,\dots,u_k$ where
\[
 u_i=e_{n-k+i-\lambda_i}+\sum_{1\leq j\leq n-k+i-\lambda_i, j\neq n-k+\ell-\lambda_\ell,\ell\leq i } x_{ij}e_j
\]
for $1\leq i\leq k$. In other words, $U$ is spanned by the rows of the matrix
\[
\begin{pmatrix}
* &\dots & * & 1& 0 & \hdotsfor{13} & 0  \\
* &\dots & * & 0 & * &\dots & * & 1 & 0 & \hdotsfor{9} & 0 \\
* & \dots & * & 0 & * &\dots & * & 0 & * &\dots & *  & 1 &\hdotsfor{6} & 0 \\
\hdotsfor{19}\\
* & \dots & * & 0 & * & \dots & * &  0 & * & \dots & * & 0 & * &\dots & * & 1 & 0 &\dots & 0
\end{pmatrix}
\]
where $1$'s are in the columns with numbers $n-k+i-\lambda_i$, $1\leq i\leq k$. Such a matrix is uniquely determined. This defines an isomorphism between $\Omega_\lambda$ and $\CC^{k(n-k)-|\lambda|}$, where $|\lambda|$ is the number of boxes in $\lambda$, and the $x_{ij}$'s are represented by stars. More precisely, this defines a system of coordinates on $\Omega_{\lambda}$ with the origin at $U^\lambda$ (for this subspace, all $x_{ij}$'s are equal to zero).

We see that $\Omega_\lambda$ is formed by the subspaces spanned by rows of matrices of the form
\[
\begin{pmatrix}
* &\dots & * & *& 0 & \hdotsfor{13} & 0  \\
* &\dots & * & * & * &\dots & * & * & 0 & \hdotsfor{9} & 0 \\
* & \dots & * & * & * &\dots & * & * & * &\dots & *  & * &\hdotsfor{6} & 0 \\
\hdotsfor{19}\\
* & \dots & * & * & * & \dots & * &  * & * & \dots & * & * & * &\dots & * & * & 0 &\dots & 0
\end{pmatrix},
\]
where the rightmost star in each row corresponds to a nonzero element. Of course, such a matrix is not uniquely determined by $U$. From this description we conclude that if $\mu\supset\lambda$, then $\Omega_\lambda\subset\overline{\Omega_\mu}$: for each $\mu\supset\lambda$, we can form a sequence of elements from $\Omega_\lambda$ whose limit belongs to $\Omega_\mu$. This means that $\Omega_\lambda\subset X_\lambda\subset\overline{\Omega_\lambda}$, and since $X_\lambda$ is closed, $X_\lambda=\overline{\Omega_\lambda}$. The proposition is proved.

\end{proof}

\begin{remark}\label{rem:finitefield} The main tool in the proof of this proposition is the Gaussian elimination (bringing a matrix to a row-echelon form by row operations). It can be carried out over an arbitrary field $\KK$, not necessarily $\CC$. This means that a Grassmannian $\Gr(k,\KK^n)$ of $k$-planes in an $n$-space over any field has a Schubert decomposition into strata isomorphic to affine spaces over $\KK$. We will use this idea later for $\KK=\FF_q$ to compute the Poincar\'e polynomial of a Grassmannian.
\end{remark}

\begin{example} For $\Gr(2,4)$, there are 6 Schubert varieties, corresponding to 6 Young diagrams inside a $2\times 2$ rectangle. The inclusion diagram of the Schubert varieties is as follows:
\[
\xymatrix{ 
& \varnothing\ar@{-}[d] \\
& {\yng(1)}\ar@{-}[dl]\ar@{-}[dr]\\
{\yng(2)}\ar@{-}[dr] &&{\yng(1,1)}\ar@{-}[dl]\\
& {\yng(2,1)}\ar@{-}[d]\\
& {\yng(2,2)}
}
\]
\end{example}

Consider the subvariety $X_{(1)}\subset\Gr(2,4)$. The points of $\Gr(2,4)$ correspond to 2-dimensional vector subspaces in $\CC^4$. They can be viewed as \emph{projective} lines in a three-dimensional projective space $\CC\PP^3$. A subspace $U$ is inside $X_{(1)}$ iff it intersects nontrivially with a given $2$-space $V_2$. This means that $X_{(1)}$ can be viewed as the set of all projective lines in $\CC\PP^3$ intersecting with a given line (namely, the projectivization of $V_2$).

Let us return to Problem~\ref{prob:schubert}. Take four lines in general position. The set of all lines intersecting each one of them defines a three-dimensional Schubert variety for a certain flag. Denote these varieties by $X_{(1)}$, $X_{(1)}'$, $X_{(1)}''$, and $X_{(1)}'''$. Each line meeting all four given lines then corresponds to a point in $X_{(1)}\cap X_{(1)}'\cap X_{(1)}''\cap X_{(1)}'''$, and we need to find the number of points in this intersection.

This can be done as follows. We have seen in Example~\ref{ex:g24} that under the Pl\"ucker embedding the Grassmannian $\Gr(2,4)$ is a quadric in $\PP^5$. Proposition~\ref{prop:schubertdecomp} implies  that under this embedding $X_{(1)}$ is the intersection of the Grassmannian with a hyperplane $p_{34}=0$. The other three Schubert varieties are translates of $X_{(1)}$, so they are hyperplane sections as well. This means that the intersection of all four Schubert varieties is the intersection of a quadric in $\PP^5$ with four generic hyperplanes. So it consists of two points.

We have solved the problem about four lines using geometric considerations. In more complicated problems it is usually more convenient to replace geometric objects by their cohomology classes, and their intersections by cup-products of these classes. We pass to the cohomology ring of the Grassmannian in the next subsection.

\subsection{Schubert classes} In this subsection we start with recalling some basic facts on homology and cohomology of algebraic varieties. 

Let $X$ be a  nonsingular projective complex algebraic variety of dimension $n$. Then it can be viewed as a $2n$-dimensional compact differentiable manifold with a canonical orientation. This gives us a canonical generator of the group $H_{2n}(X)$: the \emph{fundamental class} $[X]$. It defines the Poincar\'e pairing between the homology and cohomology groups: $H^i(X)\to H_{2n-i}(X)$, $\alpha\mapsto \alpha\cap [X]$; it is an isomorphism for all $i$. 

For each subvariety $Y\subset X$ of dimension $m$, we can similarly define its fundamental class $[Y]\in H_{2m}(Y)$. Using the Poincar\'e duality, the image of this class in $H_{2m}(X)$ defines the fundamental class $[Y]\in H^{2d}(X)$, where $d=n-m$ is the codimension of $Y$ in $X$. This can be done even for a singular $Y$ (see \cite[Appendix~A]{Manivel98} for details on singular (co)homology). In particular, the fundamental class $[x]\in H^{2n}(X)$ of a point $x\in X$ is independent of a point and generates the group $H^{2n}(X)$.

The cohomology ring $H^*(X)$ has a product structure, usually referred to as the \emph{cup product}, but we shall denote it just by a dot.  For two classes $\alpha,\beta\in H^*(X)$, let $\langle\alpha,\beta\rangle$ denote the coefficient in front of $[x]$ in the cup product $\alpha\cdot \beta$.  This defines a symmetric bilinear form on $H^*(X)$, called the \emph{Poincar\'e duality pairing}. It is nondegenerate over $\ZZ$ if $H^*(X)$ is torsion-free.


The classes of Schubert varieties $\sigma_\lambda:=[X_\lambda]\in H^{2|\lambda|}(\Gr(k,V))$ will be called \emph{Schubert classes}.

The Schubert cells $\Omega_\lambda$ form a cellular decomposition of $\Gr(k,V)$. Moreover, they are even-dimensional; this means that all differentials between the groups of cellular cocycles are zero. This means that Proposition~\ref{prop:schubertdecomp} implies the following statement.
\begin{corollary} The cohomology ring of $\Gr(k,V)$ is freely generated as an abelian group by the Schubert cycles:
\[
 H^*(\Gr(k,V),\ZZ)=\bigoplus_{\lambda\subset k\times(n-k)} \ZZ\cdot\sigma_\lambda,
\]
where $\lambda$ varies over the set of all partitions with at most $k$ rows and at most $n-k$ columns. 
\end{corollary}

Introduce the \emph{Poincar\'e polynomial} of $\Gr(k,V)$ as the generating function for the sequence of ranks of cohomology groups:
\[
 P_q(\Gr(k,V))=\sum_{k\geq 0} q^k\rk H^{2k}(\Gr(k,V)).
\]
Schubert decomposition allows us to compute the Poincar\'e polynomial of $\Gr(k,V)$.
\begin{corollary}\label{cor:poincarepolynomial} The Poincar\'e polynomial of $\Gr(k,V)$ equals
 \[
  P_q(\Gr(k,V))=\frac{(q^n-1)(q^n-q)\dots(q^n-q^{n-k+1})}{(q^k-1)(q^k-q)\dots(q^k-q^{k-1})}=\qbinom{n}{k}_q.
 \]
\end{corollary}

\begin{proof}
 Let $q=p^k$ be a power of a prime. In Exercise~\ref{xca:finitefield} we have shown that the Grassmannian $\Gr(k,\FF_q^n)$ consists of $\qbinom{n}{k}_q$ points. The same number can also be computed in a different way: as it was observed in Remark~\ref{rem:finitefield}, $\Gr(k,\FF_q^n)$ is a disjoint union of Schubert cells, each of them being isomorphic to $\FF_q^m$, where $m$ is the dimension of a Schubert cell. This means that 
all $m$-dimensional cells consist of $\rk H^{2m}(\Gr(k,V))\cdot q^m$ points, and the total number of points of the Grassmannian is nothing but the value of the Poincar\'e polynomial at $q$.
\end{proof}

\subsection{Transversality and Kleiman's theorem}\label{ssec:kleiman} Let $Y$ and $Z$ be two irreducible subvarieties of $X$ of codimensions $d$ and $d'$ respectively. The intersection of $Y$ and $Z$ is the union of several irreducible components $C_i$:
\[
 Y\cap Z=\bigcup C_i,
\]
Each of these components satisfies $\codim C_i\leq d+d'$. We shall say that $Y$ and $Z$ \emph{meet properly in $X$} if for each irreducible component of their intersection has the expected codimension: $\codim C_i=\codim Y+\codim Z$.

If $Y$ and $Z$ meet properly in $X$, then in $H^*(X)$ we have
\[
[Y]\cdot [Z]=\sum m_i [C_i],
\]
where the sum is taken over all irreducible components of the intersection, and $m_i$ is the \emph{intersection multiplicity of $Y$ and $Z$ along $C_i$}, a positive integer.  Further, this number is equal to 1 if and only if $Y$ and $Z$ \emph{intersect transversally along $C_i$}, i.e.,  a generic point $x\in C_i$ is a smooth point of $C_i$, $Y$, and $Z$ such that the tangent space to $C_i$ equals the intersection of the tangent spaces to $Y$ and $Z$:
\[
 T_x C_i= T_x Y\cap T_x Z\subset T_x X.
\]
So, if the intersection of $Y$ and $Z$ is  transversal along each component, the product of the classes $[Y]$ and $[Y']$ equals the sum of classes of the components $C_i$:
\[
 [Y]\cdot [Z]=\sum [C_i]\in H^{2d+2d'}(X).
\]

In particular, if $Y$ and $Z$ have complementary dimensions: $\dim Y+\dim Z=\dim X$, then $Y$ meets $Z$ properly iff their intersection is finite. In case of transversal intersection, this means that the Poincar\'e pairing of $[Y]$ and $[Z]$ equals the number of points in the intersection:
\[
\langle [Y],[Z]\rangle=\#(Y\cap Z).
\]

\begin{theorem}[Kleiman~\cite{Kleiman74}; cf. also~{\cite[Theorem~III.10.8]{Hartshorne77}}]\label{thm:kleiman} Let  $X$ be a homogeneous variety with respect to an algebraic group $G$.
Let  $Y$, $Z$ be subvarieties of $X$, and let $Y_0\subset Y$ and $Z_0\subset Z$ be nonempty open subsets consisting of nonsingular points. Then there exists a nonempty open subset $G_0\subset G$ such that for any $g\in\Omega$, $Y$ meets $gZ$ properly, and $Y_0\cap gZ_0$ is nonsingular and dense in $Y\cap gZ$. Thus, $[Y]\cdot [Z]=[Y\cap gZ]$ for all $g\in G_0$. 

In particular, if $\dim X=\dim Y+\dim Z$, then $Y$ and $gZ$ meet transversally for all $g\in G_0$, where $G_0\subset G$ is a nonempty open set. Thus, $Y\cap gZ$ is finite, and $\langle [Y],[Z]\rangle=\#(Y\cap gZ)$ for general $g\in G$.
\end{theorem}

\subsection{The Poincar\'e duality} Let us recall the notation from Subsection~\ref{ssec:decomp}. Let $e_1,\dots,e_n$ be a basis of $V$; as before, we fix a complete flag $V_1\subset V_2\subset\dots\subset V_n=V$, where $V_i=\langle e_1,\dots,e_i\rangle$. We also consider an \emph{opposite flag} $V_1'\subset V_2'\subset\dots\subset V_n'=V$, defined as follows: $V_i'=\langle e_{n-i+1},\dots,e_n\rangle$. To each of these flags we can associate a Schubert decomposition of the Grassmannian $\Gr(k,V)$; denote the corresponding Schubert varieties by $X_\lambda$ and $X'_\lambda$ respectively. We will refer to the latter as to an \emph{opposite Schubert variety}. Since the group $\GL(V)$ acts transitively on the set of complete flags, the class $\sigma_\lambda=[X_\lambda]=[X_\lambda']$ depends only on the partition $\lambda$ and does not depend on the choice of a particular flag.

We have seen in~\ref{ssec:decomp} that if $U\in \Omega_\lambda$, then it admits a unique basis $u_1,\dots,u_k$ such that the coefficients of decomposition of $u_i$'s with respect to the basis $e_1,\dots,e_n$ form a matrix 
\[
\begin{pmatrix}
* &\dots & * & 1& 0 & \hdotsfor{13} & 0  \\
* &\dots & * & 0 & * &\dots & * & 1 & 0 & \hdotsfor{9} & 0 \\
* & \dots & * & 0 & * &\dots & * & 0 & * &\dots & *  & 1 &\hdotsfor{6} & 0 \\
\hdotsfor{19}\\
* & \dots & * & 0 & * & \dots & * &  0 & * & \dots & * & 0 & * &\dots & * & 1 & 0 &\dots & 0
\end{pmatrix}
\]
where the $1$ in the $i$'th line occurs in the column number $n-k+i-\lambda_{i}$.

Let $\mu$ be another partition. Consider a subspace $W\in \Omega_\mu'$ from the Schubert cell corresponding to $\mu$ and the flag $V_\bullet'$. A similar reasoning shows that such a subspace is spanned by the rows of a matrix
\[
\begin{pmatrix}
0 & \dots & 0 & 1 & * & \dots & * &  0 & * & \dots & * & 0 & * &\dots & * & 0 & * &\dots & *\\
\hdotsfor{19}\\
0 &\hdotsfor{6}  0 & 1& * &\dots & * & 0 &* &\dots & * & 0 & * &\dots & *\\
0 &\hdotsfor{9} & 0 & 1& * &\dots & * & 0 & * &\dots & *\\
0 & \hdotsfor{13} & 0 & 1 & * &\dots & *\\
\end{pmatrix}
\]
where the $1$ in the $i$'th line is in the column $\mu_{k+1-i}+i$.

Suppose that a $k$-space $U$ belongs to the intersection $\Omega_\lambda\cap \Omega_\mu'$. This means that it admits two bases of such a form simultaneously. In particular, this means that for each $i$ the leftmost nonzero entry in the $i$-th line of the first matrix non-strictly precedes the rightmost nonzero entry in the $i$-th row of the second matrix, which means that $\mu_{k+1-i}+i\leq n-k+i-\lambda_i$, or, equivalently, $\mu_{k+1-i}+\lambda_i\leq n-k$. This means that if $\Omega_\lambda\cap \Omega_\mu'\neq\varnothing$, then the diagram $\mu$ is contained in the \emph{complement} $\widehat\lambda$ to the diagram $\lambda$.

Denote by $\delta_{\mu,\widehat\lambda}$ the Kronecker symbol, which is equal to 1 if $\mu=\widehat\lambda$ and to 0 otherwise.

\begin{prop}\label{prop:duality}
 Let $\lambda$ and $\mu$ be two partitions contained in the rectangle of size $k\times (n-k)$, and let $|\lambda|+|\mu|=k(n-k)$. Then
\[
 \sigma_\lambda\cdot\sigma_\mu=\delta_{\mu,\widehat\lambda}.
\]
\end{prop}

\begin{proof} According to the previous discussion, if $|\lambda|+|\mu|=k(n-k)$, then $X_\lambda\cap X_\mu'=\Omega_\lambda\cap \Omega_\mu'$. Indeed, from the inclusion relations on Schubert varieties we conclude that if there were a point $U\in X_\lambda\setminus \Omega_\lambda$, $U\in X_\mu'$,  this would mean that $\Omega_\lambda'\cap \Omega_\mu'\neq \varnothing$ for some $\lambda'\subsetneq\lambda$ and $\mu'\subseteq \mu$, which is nonsense, because $|\lambda'|+|\mu'|<k(n-k)$.

If the dimensions of $X_\lambda$ and $X_\mu'$ add up to $k(n-k)$, the intersection is nonzero only if the diagrams $\lambda$ and $\mu$ are complementary. In this case the intersection $\Omega_\lambda\cap \Omega_\mu'$ is easy to describe: it is a unique point $U^\lambda=\langle e_{n-k+1-\lambda_1},\dots,e_{n-\lambda_k}\rangle$. It is also clear that this intersection is transversal, because in the natural coordinates in the neighborhood of this point the tangent spaces to $\Omega_\lambda$ and $\Omega_{\mu}'$ are coordinate subspaces spanned by two disjoint sets of coordinates.
\end{proof}

\subsection{Littlewood--Richardson coefficients} In the previous subsection we were studying the intersection of two Schubert varieties  $X_\lambda$ and $X_\mu'$ of complementary dimension. Kleiman's transversality theorem shows what happens if the dimensions of $X_\lambda$ and $X_\mu'$ are arbitrary.

First let us find out when such an intersection is nonempty. This can be done by essentialy the same argument as in the proof of Proposition~\ref{prop:duality}, so we leave it as an exercise to the reader.

\begin{xca} Show that the intersection $X_\lambda\cap X_\mu'$ is nonempty iff $\lambda\subseteq \widehat\mu$.
\end{xca}

Kleiman's transversality theorem implies that the intersection  $X_\lambda\cap X_\mu'$ is proper. Indeed, it states that there exists a nonempty open set $G_0\subset\GL(V)$ such that $X_\lambda$ intersects  $g X_\mu'$ properly for all $g\in G_0$. 

Further, a classical fact from linear algebra states that a generic element $g\in \GL(V)$ can be presented as $g=b\cdot b'$, where $b$ and $b'$ are given by an upper-triangular and lower-triangular matrices respectively (this is sometimes called LU-decomposition, but essentially this is nothing but Gaussian elimination). This means that there exists an element $g\in G_0$ also admitting such a decomposition.

The elements $b$ and $b'$ stabilize the flags $V_\bullet$ and $V_\bullet'$; so the varieties $X_\lambda$ and $X_\mu'$ are also $b$- and $b'$-invariant. This means that $X_\lambda$ intersects $bb'X_\mu'=b X_\mu'$ properly. Shifting both varieties by $b^{-1}$, we obtain the desired result.

In fact, a stronger result holds; see \cite{BrionLakshmibai03} for details.

\begin{prop} The intersection $X_\lambda^\mu:=X_\lambda\cap X_\mu'$, if nonempty, is an irreducible variety, called a \emph{Richardson variety}. Its codimension is given by $\codim X_\lambda^\mu=|\widehat \mu|-|\lambda|$.
\end{prop}

So in the cohomology ring $H^*(\Gr(k,V))$ we have $\sigma_\lambda\cdot\sigma_\mu=[X_\lambda]\cdot [X_\mu']=[X_\lambda^\mu]$. Together with the Poincar\'e duality (Proposition~\ref{prop:duality}) and Kleiman's transversality this implies the following theorem.

\begin{theorem}\label{thm:LR}
\begin{enumerate} 
\item For any subvariety $Z\subset \Gr(k,V)$, we have 
\[
[Z]=\sum a_\lambda\sigma_\lambda,
\]
where $a_\lambda=\langle [Z], \sigma_{\widehat\lambda}\rangle=\#(Z\cap g X_{\widehat\lambda})$ for general $g\in \GL(V)$. In particular, the coefficients $a_\lambda$ are nonnegative.

\item Let the coefficients $c_{\lambda\mu}^\nu$ be the structure constants of the ring $H^*(\Gr(k,V)$, defined by
\[
 \sigma_\lambda\cdot\sigma_\mu=\sum_\nu c^\nu_{\lambda\mu}\sigma_\nu.
\]
Then $c_{\lambda\mu}^\nu$ are nonnegative integers.
\end{enumerate}
\end{theorem}


The integers $c^\nu_{\lambda\mu}$ are called \emph{the Littlewood--Richardson coefficients}. Note that they only can be nonzero if $|\lambda|+|\mu|=|\nu|$.

This result is essentially geometric. But it also leads to a very nontrivial combinatorial problem: to give these coefficients a combinatorial meaning. Such an interpretation, known as \emph{the Littlewood--Richardson rule}, was given by Littlewood and Richardson \cite{LittlewoodRichardson34} in 1934: they claimed that the number $c_{\lambda\mu}^\nu$ were equal to the number of skew semistandard Young tableaux of shape $\nu/\lambda$ and weight $\mu$ satisfying certain combinatorial conditions. However, they only managed to prove it in some simple cases. The first rigorous proof was given by M.-P.~Sch\"utzenberger more than 40 years later \cite{Schuetzenberger77}; it used combinatorial machinery developed by Schensted, Knuth and many others. 

There are other interpretations of the Littlewood--Richardson rule. Some of them imply symmetries of Littlewood--Richardson coefficients (such as symmetry in $\lambda$ and $\mu$), which are not obvious from the original description; in particular, let us mention the paper by V.~Danilov and G.~Koshevoy about \emph{massifs} \cite{DanilovKoshevoy05} and a very nice construction by Knutson, Tao and Woodward  \cite{KnutsonTaoWoodward04} interpreting the Littlewood--Richardson coefficients as the numbers of \emph{puzzles}. A good survey on puzzles can be found, for instance, in \cite{CoskunVakil09}. The Littlewood--Richardson rule was also generalized to the much more general case of complex senisimple Lie algebras by Littelmann in \cite{Littelmann94}; this interpretation involved the so-called \emph{Littelmann paths}.

We won't speak about the Littlewood--Richardson rule in general; the reader can refer to \cite{Fulton97} or to \cite{Manivel98}. The Poincar\'e duality is one of its particular cases. Further we will only deal with one more particular case, when $X_\mu$ is a so-called \emph{special Schubert variety}, corresponding to a one-row or a one-column diagram. This situation is governed by the \emph{Pieri rule}.

\subsection{Pieri rule for Schubert varieties}\label{ssec:pieri} Here is one more special case of the Littlewood--Richardson rule. Let $(m)$ be a one-line partition consisting of $m$ boxes. We will describe the rule for multiplying the class $\sigma_m$ by an arbitrary Schubert class $\sigma_\lambda$. The Schubert varieties $X_{(m)}$ corresponding to one-line partitions are usually called \emph{special Schubert varieties}.

Let us introduce some notation. Let $\lambda$ be an arbitrary partition. Denote by $\lambda\otimes m$ the set of all partitions obtained from $\lambda$ by adding $m$ boxes in such a way that no two added boxes are in the same column.

\begin{example}
 Let $\lambda=(3,2)$, $m=2$. The elements of the set $\lambda\otimes m$ are listed below. The added boxes are marked by stars.
\[
 \young(~~~**,~~)\qquad \young(~~~*,~~*)\qquad \young(~~~,~~*,*) \qquad\young(~~~*,~~,*)\qquad \young(~~~,~~,**)
\]
\end{example}

We have seen that the Schubert classes $\sigma_\lambda$ and $\sigma_{\widehat\lambda}$ are dual. That is, if $\alpha\in H^*(\Gr(k,V))$, then
\[
 x=\sum_{\lambda\subset k\times(n-k)} \langle \alpha, \sigma_{\widehat\lambda}\rangle\sigma_\lambda.
\]
\begin{theorem}[Pieri rule] Let $\lambda\subset k\times (n-k)$ be a partition, and $m\leq n-k$. Then
\[
 \sigma_\lambda\cdot\sigma_m=\sum_{\nu\in k\times(n-k), \nu\in\lambda\otimes m} \sigma_\nu.
\] 
\end{theorem}

\begin{proof} It is enough to show that if $|\lambda|+|\mu|=k(n-k)-m$, then $\sigma_\lambda\sigma_\mu\sigma_m=1$ if the condition
\[
 n-k-\lambda_k\geq\mu_1\geq n-k-\lambda_{k-1}\geq\mu_2\geq\dots\geq n-k-\lambda_1\geq\mu_k
\]
holds, and $\sigma_\lambda\sigma_\mu\sigma_m=0$ otherwise. So we have a necessary condition: $\lambda_i+\mu_{n-k+1-i}\leq  n-k$ for each $i$, otherwise $\sigma_\lambda\sigma_\mu=0$. Let us set
\begin{eqnarray*}
 A_i & =&\langle e_1,\dots,e_{n-k+i-\lambda_i}\rangle =  V_{n-k+i-\lambda_i},\\
 B_i&=&\langle e_{\mu_{k+1-i}+i},\dots,e_n\rangle = V'_{n+1-i-\mu_{k+1-i}},\\
 C_i&=&\langle e_{\mu_{k+1-i}+i},\dots,e_{n-k+i-\lambda_i}\rangle = A_i\cap B_i.
\end{eqnarray*}
The above condition holds if and only if the subspaces $C_1,\dots,C_k$  form a direct sum, i.e., if their sum $C=C_1+\dots+C_k$ has dimension $k+m$. Note that $C=\cap_i (A_i+B_i)$.

If $U\in X_\lambda\cap X_\mu'$, we have $\dim(U\cap A_i)\geq i$ and $\dim (U\cap B_i)\geq k+1-i$. This means that for each $i$ we have $U\subset A_i+B_{i+1}$. Indeed, if this sum is not equal to the whole space $V$, we conclude that $A_i$ and $B_{i+1}$ form a direct sum, and so
\[
 \dim (U\cap (A_i+B_{i+1}))\geq i+(k-i)=k.
\]
So $U\subset C$.

Let $L$ be a subspace of $V$ of dimension $n-k+1-m$. Consider the associated Schubert  variety
\[
 X_m(L)=\{U\in\Gr(k,V),U\cap L\neq 0\}.
\]
If the above condition does not hold, then $\dim C\leq n-k+m-1$, and we can choose $L$ intersecting $L$ trivially. This would mean that $X_\lambda\cap X_\mu'\cap X_m(L)=\varnothing$, and $\sigma_\lambda\sigma_\mu\sigma_m=0$.

In the opposite case, if $\dim C=k+m$, the intersection of $C$ with a generic subspace of dimension $n-k+1-m$ is a line $\langle u\rangle\in C$. Let $u=u_1+\dots+u_k$, where $u_i\in C_i$ (recall that this sum is direct). All the $u_i$'s are necessarily in $U$, and they are linearly independent, so they form a basis of $U$. Thus the intersection of $X_\lambda\cap X_\mu'\cap X_m(L)$ is a point. A standard argument, similar to the one used in the proof of Proposition~\ref{prop:duality}, shows that this intersection is transversal, so $\sigma_\lambda\sigma_\mu\sigma_m=1$.
\end{proof}

\begin{example}\label{ex:schubert_pieri}
The Pieri rule allows us to solve our initial problem using Schubert calculus. As we discussed, we would like to find the 4-th power of the class $\sigma_1\in H^*(\Gr(2,4))$. Using the Pieri rule, we see that:
\[
 \sigma_1^2=\sigma_2+\sigma_{(1,1)},
\]
since one box can be added to a one-box diagram in two different ways:
\[
\young(~)\otimes 1=\left\{\young(~*),\qquad\young(~,*)\right\} 
\]
Then,
\[
 \sigma_1^3=\sigma_1(\sigma_2+\sigma_{(1,1)})=2\sigma_{(2,1)},
\]
since the two other diagrams $\young(~~*)$ and $\young(~,~,*)$ do not fit inside the $(2\times 2)$-box and thus are not counted in the Pieri rule. Finally, we multiply the result by $\sigma_1$ for the fourth time and see that
\[
 \sigma_1^4=2\sigma_{(2,2)}=2[pt].
\]
So there are exactly two lines meeting four given lines in general position.
\end{example}

We can look at the same problem in a slightly different way: if we consider the Grassmannian $\Gr(2,4)$ as a subset of $\PP^5$ defined by the Pl\"ucker embedding, the cycle $\sigma_1$ corresponds to its hyperplane section. This means that $\sigma_1^4$ equals the class of a point times the number of points in the intersection of $\Gr(2,4)$ with four generic hyperplanes, i.e., the degree of the Grassmannian (and we have already seen that $\Gr(2,4)$ is a quadric). So in the above example we have used the Pieri rule to compute the degree of $\Gr(2,4)$ embedded by Pl\"ucker. 

This can be easily generalized for the case of an arbitrary Schubert variety in an arbitrary Grassmannian. 

\subsection{Degrees of Schubert varieties}\label{ssec:degrees} In this subsection we will find the degrees of Schubert varieties and in particular of the Grassmannian under the Pl\"ucker embedding. For this first let us recall the notion of a \emph{standard Young tableau}.

\begin{definition} Let $\lambda$ be a Young diagram consisting of $m$ boxes. A \emph{standard Young tableau} of shape $\lambda$ is a (bijective) filling of the boxes of $\lambda$ by the numbers $1,\dots,m$ such that the numbers in the boxes increase by rows and by columns. We will denote the set of standard Young tableaux of shape $\lambda$ by $SYT(\lambda)$.
\end{definition}

\begin{example} Let $\lambda=(2,2)$; then there are two standard tableaux of shape $\lambda$, namely,
\[
\young(12,34)\quad\text{and}\quad\young(13,24).
\]
\end{example}


\begin{theorem} The degree of a Schubert variety $X_\lambda\subset\Gr(k,V)\subset \PP(\Lambda^k V)$ is equal to the number of standard Young tableaux of shape $\widehat\lambda$, where $\widehat{\lambda}$ is the complementary diagram to $\lambda$ in the rectangle of size $k\times(n-k)$ and $n=\dim V$.
\end{theorem}

\begin{proof} By definition, the degree of an $m$-dimensional variety $X\subset\PP^N$ in a projective space is the number of points in the intersection of $X$ with $m$ hyperplanes in general position. 

Proposition~\ref{prop:schubertdecomp} implies that a hyperplane section of a Grassmannian under the Pl\"ucker embedding corresponds to the first special Schubert variety $X_{(1)}$, or, on the level of cohomology, to the class $\sigma_{1}$.

This means that if $\dim X_\lambda=m$ and $\deg X_\lambda=d$, then
\[
\sigma_\lambda\cdot\sigma_{1}^m=d\cdot[pt].
\]

This allows us to compute $d$ using the Pieri rule: $d$ is the number of ways to obtain a rectangle of size $k\times (n-k)$ from $\lambda$ by adding $m$ \emph{numbered} boxes, and those ways are in an obvious bijection with the standard Young tableaux of shape $\widehat{\lambda}$.
\end{proof}

The number of standard Young tableaux can be computed via the \emph{hook length formula}, due to Frame, Robinson, and Thrall. Let $s\in\lambda$ be a box of a Young diagram $\lambda$; the \emph{hook} corresponding to $s$ is the set of boxes below \emph{or} to the right of $s$, including $s$ itself. An example of a hook is shown on the figure below. Let us denote the number of boxes in the hook corresponding to $s$ by $h(s)$.
\[
\young(~~~~~~~,~~s***,~~*~,~~*,~~)
\]

\begin{theorem}[Frame--Robinson--Thrall, \cite{FrameRobinsonThrall54}] The number of standard Young tableaux of shape $\lambda$ is equal to
\[
\# SYT(\lambda)=\frac{|\lambda|!}{\prod_{s\in\lambda} h(s)},
\]
where the product in the denominator is taken over all boxes $s\in\lambda$.
\end{theorem}

This formula has several different proofs; some of them can be found in \cite[Sec.~1.4.3]{Manivel98} or \cite{Fulton97}.

As a corollary, we get the classical result due to Schubert on the degree of the Grassmannian, which we have already mentioned in the introduction (with a slightly different notation).

\begin{corollary} The degree of a Grassmannian $\Gr(k,V)\subset\PP\Lambda^k V$ under the Pl\"ucker embedding equals 
\[
\deg \Gr(k,V)=(k(n-k))!\frac{0!\cdot 1!\cdot \dots\cdot (k-1)!}{(n-k)!\cdot(n-k)!\cdot\dots\cdot  (n-1)!}
\]
\end{corollary}

\begin{xca} Deduce this corollary from the hook length formula.
\end{xca}

\subsection{Schur polynomials}

In the remaining part of this section we reinterpret questions on the intersection of Schubert varieties in terms of computations in a quotient ring of the ring of symmetric polynomials. For this let us first recall some facts about symmetric and skew-symmetric polynomials.

Let $\Lambda_k=\ZZ[x_1,\dots,x_k]^{S_k}$ be the ring of symmetric polynomials. Denote by $e_m$ and $h_m$ the $m$-th \emph{elementary symmetric polynomial} and \emph{complete symmetric polynomial}, respectively:
\[
e_m=\sum_{1\leq i_1<\dots<i_m\leq k} x_{i_1}\dots x_{i_m}\qquad\text{and}\qquad
h_m=\sum_{1\leq i_1\leq \dots\leq i_m\leq k} x_{i_1}\dots x_{i_m}.
\]
In particular, $e_1=h_1=x_1+\dots+x_k$, $e_k=x_1\dots x_k$, and $e_m=0$ for $m>k$ (while all $h_m$ are nonzero).

The following theorem is well-known.
\begin{theorem}[Fundamental theorem on symmetric polynomials] Each of the sets $e_1,\dots,e_k$ and $h_1,\dots,h_k$ freely generates the ring of symmetric polynomials:
\[
\Lambda_k=\ZZ[e_1,\dots,e_k]=\ZZ[h_1,\dots,h_k].
\]
\end{theorem}

This theorem means that all possible products $e_1^{i_1}\dots e_k^{i_k}$ for $i_1,\dots,i_k\geq 0$ form a basis of $\Lambda_k$ as a $\ZZ$-module, and so do the elements  $h_1^{i_1}\dots h_k^{i_k}$. But now we will construct another basis of this ring, which is more suitable for our needs. Its elements will be called \emph{Schur polynomials}.

For this consider the set of skew-symmetric polynomials, i.e., the polynomials satisfying the relation
\[
f(x_1,\dots,x_k)=(-1)^\sigma f(x_{\sigma(1)},\dots,x_{\sigma(k)}), \qquad \sigma\in S_k.
\]
They also form a $\ZZ$-module (and also a $\Lambda_k$-module, but not a subring in $\ZZ[x_1,\dots,x_k]$). Let us construct a basis of this module indexed by partitions $\lambda$ with at most $k$ rows: for each $\lambda=(\lambda_1,\dots,\lambda_k)$, where $\lambda_1\geq \dots\geq\lambda_k\geq 0$, let us make this sequence into a strictly increasing one by adding $k-i$ to its $i$-th term:
\[
\lambda+\delta=(\lambda_1+k-1,\lambda_2+k-2,\dots,\lambda_{k-1}+1,\lambda_k).
\]
Now consider a skew-symmetric polynomial $a_{\lambda+\delta}$ obtained by skew-symmetrization from  $x^{\lambda+\delta}:=x_1^{\lambda_1+k-1}x_2^{\lambda_2+k-2}\dots x_k^{\lambda_k}$:
\[
a_{\lambda+\delta}=\sum_{\sigma\in S_k}(-1)^\sigma x_{\sigma(1)}^{\lambda_1+k-1}x_{\sigma(2)}^{\lambda_2+k-2}\dots x_{\sigma(k)}^{\lambda_k}.
\]
This polynomial can also be presented as a determinant
\[
a_{\lambda+\delta}=\begin{vmatrix}
x_1^{\lambda_1+k-1} &x_2^{\lambda_1+k-1}& \dots & x_k^{\lambda_1+k-1}\\
x_1^{\lambda_2+k-2} &x_2^{\lambda_2+k-2}& \dots & x_k^{\lambda_2+k-2}\\
\vdots & \vdots & \ddots &\vdots\\
x_1^{\lambda_k} &x_2^{\lambda_k}& \dots & x_n^{\lambda_k}\\
\end{vmatrix}
\]

Every symmetric polynomial is divisible by $x_i-x_j$ for each $i<j$. This means that $a_{\lambda+\delta}$ is divisible by the \emph{Vandermonde determinant} $a_\delta$ corresponding to the empty partition:
\[
a_\delta=\prod_{i>j} (x_i-x_j)=
\begin{vmatrix}
x_1^{k-1} &x_2^{k-1}& \dots & x_k^{k-1}\\
x_1^{k-2} &x_2^{k-2}& \dots & x_k^{k-2}\\
\vdots & \vdots & \ddots &\vdots\\
1 &1& \dots &1\\
\end{vmatrix}
\]

\begin{definition} Let $\lambda$ be a partition with at most $k$ rows. Define the \emph{Schur polynomial} corresponding to $\lambda$ as the quotient
\[
s_\lambda(x_1,\dots,x_k)=a_{\lambda+\delta}/a_\delta.
\]
\end{definition}

\begin{xca} Show that if $\lambda=(m)$ is a one-line partition, the corresponding Schur polynomial is equal to the $k$-th complete symmetric polynomial: $s_{(m)}=h_m$. Likewise, if $\lambda=(1^m)$ is a one-column partition formed by $m$ rows of length 1, then $s_{(1^k)}=e_m$ is the $m$-th elementary symmetric polynomial.
\end{xca}

Schur polynomials also admit a combinatorial definition (as opposed to the previous algebraic definition). It is based on the notion of Young tableaux, which we have already seen in the previous subsection. Let $\lambda$ be a partition with at most $k$ rows. A \emph{semistandard Young tableau of shape $\lambda$} is a filling of the boxes of $\lambda$ by integers from the set $\{1,\dots,k\}$ in such a way that the entries in the boxes non-strictly increase along the rows and strictly increase along the columns. Denote the set of all semistandard Young tableaux of shape $\lambda$ by $SSYT(\lambda)$. Let $T$ be such a tableau; denote by $x^T$ the monomial $x_1^{t_1}\dots x_k^{t_k}$, where $t_1,\dots,t_k$ are the numbers of occurence of the entries $1,\dots,k$ in $T$.

The following theorem says that the Schur polynomial $s_\lambda$ is obtained as the sum of all such $x^T$ where $T$ runs over the set of all semistandard Young tableaux of shape $\lambda$.
\begin{theorem}\label{thm:combschur} Let $\lambda$ be a Young diagram with at most $k$ rows. Then
 \[
  s_\lambda(x_1,\dots,x_k)=\sum_{T\in SSYT(\lambda)} x^T.
 \]
\end{theorem}

We will not prove this theorem here; its proof can be found in \cite{Manivel98} or in \cite{Fulton97}.

\begin{example}
 Let $k=3$, $\lambda=(2,1)$. There are 8 semistandard Young tableaux of shape $\lambda$:
\[
 \young(11,2) \qquad\young(11,3) \qquad\young(12,2) \qquad\young(12,3) \qquad\young(13,2)\qquad \young(13,3)\qquad \young(22,3)\qquad \young(23,3)
\]
The corresponding Schur polynomial then equals
\[
 s_{(2,1)}(x,y,z)=x^2y+x^2z+xy^2+2xyz+xz^2+y^2z+yz^2.
\]
\end{example}

\begin{xca} Show by a direct computation that the algebraic definition of $s_{(2,1)}(x,y,z)$ gives the same result.
\end{xca}

\begin{remark} Theorem~\ref{thm:combschur} provides an easy way to compute Schur polynomials (this is easier than dividing one skew-symmetric polynomial by another). However, this theorem is by no means trivial: first of all, it is absolutely not obvious why does the summation over all Young tableaux of a certain shape give a symmetric polynomial! We will see an analogue of this theorem for flag varieties (Theorem~\ref{thm:fominkirillov}), but there Young tableaux will be replaced by more involved combinatorial objects, \emph{pipe dreams}.
\end{remark}

\subsection{Pieri rule for symmetric polynomials}
Now let us multiply a Schur polynomial by a complete or elementary symmetric polynomial. It turns out that they satisfy the same Pieri rule as Schubert classes. Recall that in Subsection~\ref{ssec:pieri} we introduced the following notation: if $\lambda$ is a Young diagram, then $\lambda\otimes 1^m$ and $\lambda\otimes m$ are two sets of diagrams obtained from $\lambda$ by adding $m$ boxes in such a way that no two boxes are in the same column (resp. in the same row).

\begin{theorem}[Pieri formulas]\label{thm:pieri_schur} With the previous notation,
\[
 s_\lambda e_m=\sum_{\mu\in\lambda\otimes 1^m} s_\mu\qquad\text{and}\qquad
s_\lambda h_m=\sum_{\mu\in\lambda\otimes m} s_\mu
\]
\end{theorem}

\begin{proof} The first formula is obtained from the identity
\begin{multline*}
 a_{\lambda+\delta} e_m=\sum_{\sigma\in S_k}\sum_{i_1<\dots<i_m} (-1)^\sigma x^{\sigma(\lambda+\delta)} x_{\sigma(i_1)}\dots x_{\sigma(i_m)}=\sum_{\alpha\in\{0,1\}^k} a_{\lambda+\alpha+\delta},
\end{multline*}
taking into account that $a_{\lambda+\alpha+\delta}$ is nonzero iff $\lambda+\alpha$ is a partition. The second formula is obtained in a similar way. 
\end{proof}

So Pieri formulas hold both for $\Lambda_k$ and $H^*(\Gr(k,n))$. Since $h_1,h_2,\dots$ and $\sigma_1,\dots,\sigma_{n-k}$ are systems of generators of those rings, they completely determine structure constants of these rings. This implies the following theorem.

\begin{theorem} The map
\[
\varphi_{k,n}\colon \Lambda_k\to H^*(\Gr(k,n)),
\]
which sends $s_\lambda$ to $\sigma_{\lambda}$ if $\lambda\subset k\times (n-k)$ and to 0 otherwise, is a ring epimorphism.
\end{theorem}

\section{Flag varieties}\label{sec:flags}

\subsection{Definition and first properties}

As before, let $V$ be an $n$-dimensional vector space. A \emph{complete flag} $U_\bullet$ in $V$ is an increasing sequence of subspaces, such that the dimension of the $i$-th subspace is equal to $i$:
\[
U_\bullet=(U_0\subset U_1\subset \dots\subset U_{n-1}\subset U_n= V),\qquad \dim U_i=i,\quad i\in [0,n].
\]

The set of all complete flags in $V$ will be denoted by $\Fl(V)$ or $\Fl(n)$.

To each basis $u_1,\dots, u_n$ of $V$ we can assign a complete flag by setting $U_i=\langle u_1,\dots,u_i\rangle$. Since $\GL(V)$ acts transitively on bases, it also acts transitively on flags. It is easy to describe the stabilizer of this action, i.e., the subgroup fixing a given flag  $U_\bullet$. Suppose that $U_\bullet$ corresponds to the standard basis $e_1,\dots,e_n$ of $V$. Then $\Stab_{\GL(V)} U_\bullet$ is the group of nondegenerate upper-triangular matrices, which we denote by $B$.

This means that $\Fl(V)=\GL(V)/B$ is a homogeneous space: each flag can be thought of as a coset of the right action of $B$ on $\GL(V)$. From this we see that $\dim\Fl(V)=\dim\GL(V)-\dim B=\frac{n(n-1)}2$. So, by the same argument as in the case of Grassmannians, it is a quasiprojective algebraic variety (or a smooth manifold, if we prefer to work with Lie groups).

There is an obvious embedding $\Fl(V)\hookrightarrow \Gr(1,V)\times\Gr(2,V)\times\dots\times\Gr(n-1,V)$ of a flag variety into a product of Grassmannians: each flag is mapped into the set of subspaces it consists of, and $\Fl(V)$ is defined inside this direct product by incidence relations $V_i\subset V_{i+1}$. If we embed each Grassmannian by Pl\"ucker into a projective space: $\Gr(k,V)\hookrightarrow \PP^{N_k-1}$, these relations will be given by algebraic equations. So, $\Fl(V)$ is an algebraic subvariety of $\PP^{N_1-1}\times\dots\PP^{N_{n-1}-1}$, where $N_k=\binom{n}{k}$.
The latter product of projective spaces can be embedded by Segre into $\PP^{N_1\dots N_{n-1}-1}$. 

Summarizing, we get the following
\begin{prop} $\Fl(V)$ is a projective algebraic variety of dimension $n(n-1)/2$.
\end{prop}

\subsection{Schubert decomposition and Schubert varieties} In this subsection we construct a decomposition of a full flag variety. It will be very similar to the Schubert decomposition of Grassmannians which we saw in the previous section.

As in the case of Grassmannians, let us fix a standard basis $e_1,\dots,e_n$ of $V$ and a complete flag related to this basis: $V_\bullet$, formed by the subspaces $V_i=\langle e_1,\dots,e_i\rangle$.
This flag is stabilized by the subgroup $B$ of nondegenerate upper-triangular matrices.

Let $w\in S_n$ be a permutation. We can associate to it the \emph{rank function} $r_w\colon \{1,\dots,n\}\times \{1,\dots,n\}\to \ZZ_{\geq 0}$ as follows:
\[
 r_w(p,q)=\#\{i\leq p, w(i)\leq q\}.
\]
This function can also be described as follows. Let $M_w$ be a \emph{permutation matrix} corresponding to $w$, i.e. the matrix whose $(i,j)$-th entry is equal to 1 if $w(i)=j$, and to 0 otherwise. Then $M_w$ permutes the basis vectors $e_1,\dots,e_n$ as prescribed by $w^{-1}$. Then $r_w(p,q)$ equals the rank of the corner submatrix of $M_w$ formed by its first $p$ rows and $q$ columns.

Define \emph{Schubert cells} $\Omega_w$ and \emph{Schubert varieties} $X_w$ as follows:
\begin{eqnarray*}
  \Omega_w&=&\{U_\bullet\in\Fl(n)\mid \dim(W_p\cap V_q)=r_w(p,q), 1\leq p,q\leq n\},\\
  X_w&=&\{U_\bullet\in\Fl(n)\mid \dim(W_p\cap V_q)\geq r_w(p,q), 1\leq p,q\leq n\}.
\end{eqnarray*}
It is clear that $X_w=\overline{\Omega_w}$.

As in the case of Schubert cells in Grassmannians, we can find a ``special point'' $U^w_\bullet$ inside each $\Omega_w$. It it stable under the action of the diagonal torus, and each of the subspaces $U_i^w$ is spanned by basis vectors:
\[
 U_i^w=\langle e_{w(1)},\dots,e_{w(n)}\rangle.
\]

Imitating the proof of Proposition~\ref{prop:schubertdecomp}, we can see that for each element $U_\bullet\in\Omega_w$ there is a uniquely determined matrix $(x_{ij})_{1\leq i,j\leq n}$ such that $U_i$ is generated by its first $i$ rows, and
\[
 x_{i,w(i)}=1\qquad\text{and}\qquad x_{ij}=0 \text{ if }j>w(i) \text{ or }i>w^{-1}(j).
\]

This matrix can be constructed as follows. We put $1$'s at each $(i,w(i))$. Then we draw a hook of zeroes going right and down from each entry filled by $1$. All the remaining entries are filled by stars (i.e., they can be arbitrary). Again we get a coordinate system on $\Omega_w$ with the origin at $U^w_\bullet$.

\begin{example} Let $w=(25413)$ (we use the one-line notation for permutations: this means that $w(1)=2$, $w(2)=5$, etc.). Then each element of $\Omega_{w}$ corresponds to a uniquely determined matrix of the form
\[
 \begin{pmatrix}
  * & 1 & 0 & 0 & 0\\
  * & 0 & * & * & 1\\
  * & 0 & * & 1 & 0\\
  1 & 0 & 0 & 0 & 0\\
  0 & 0 & 1 & 0 & 0
 \end{pmatrix}
\]
\end{example}

\begin{xca}\label{ex:length}
 Show that the number of stars is equal to the length $\ell(w)$ of the permutation $w$, i.e. the number of its inversions:
\[
 \ell(w)=\#\{(i,j)\mid i<j, w(i)>w(j)\}.
\]
\end{xca}
We have thus shown that $\Omega_w\cong\CC^{\ell(w)}$ is indeed a cell, that is, an affine space. 
Another way of proving this was to note that each $\Omega_w$ is an orbit of the left action of the upper-triangular subgroup $B$ on $\Fl(n)$, so Schubert decomposition is just the decomposition of $\Fl(n)$ into $B$-orbits.

\begin{example} Just as in the case of Grassmannians, there is a unique zero-dimensional cell, corresponding to the identity permutation $e\in S_n$, and a unique open cell $\Omega_{w_0}$ corresponding to the \emph{maximal length permutation} $w_0=(n,n-1,\dots,2,1)$.
\end{example}

\begin{definition} Let us introduce a partial order on the set of permutations $w\in S_n$: we will say that $v\leq w$ if $r_v(p,q)\geq r_w(p,q)$ for each $1\leq p,q\leq n$. This order is called the \emph{Bruhat order}.
\end{definition}

\begin{xca} Show that the permutations $e$ and $w_0$ are the minimal and the maximal elements for this order.
\end{xca}

\begin{example} This is the Hasse diagram of the Bruhat order for the group $S_3$. The edges represent \emph{covering relations}, i.e., $v$ and $w$ are joined by an edge, with $w$ on the top, if $v\leq w$ and there is no $u\in S_n$ such that $v\lneqq u\lneqq w$.
\[
 \xymatrix{
 & (321)\ar@{-}[dl]\ar@{-}[dr]\\
(312)\ar@{-}[d]\ar@{-}[drr] & & (231)\ar@{-}[d]\ar@{-}[dll]\\
(132)\ar@{-}[dr] && (213)\ar@{-}[dl]\\
&(123)
}
\]
\end{example}
For flag varieties this diagram plays the same role as the inclusion graph of Young diagrams for Grassmannians.

\begin{prop} For each permutation $w\in S_n$ its Schubert variety
\[
 X_w=\bigsqcup_{v\leq w} \Omega_v
\]
is the disjoint union of the Schubert cells of permutations that are less than or equal to $w$ with respect to the Bruhat order.
\end{prop}

\begin{xca} Prove this proposition.
\end{xca}

\begin{corollary} We have the inclusion $X_v\subset X_w$ iff $v\leq w$.
\end{corollary}

\subsection{The cohomology ring of $\Fl(n)$ and Schubert classes}

The Schubert decomposition allows us to compute the cohomology ring of $\Fl(n)$. From the cellular decomposition of $\Fl(n)$ we see that $H^*(\Fl(n))$ is generated (as an abelian group) by the cohomology classes dual to the fundamental classes of Schubert varieties. Let us perform a twist by the longest element $w_0\in S_n$ and denote by $\sigma_w$  the class dual to the fundamental class of $X_{w_0w}$.

\begin{prop} The (integer) cohomology ring of $\Fl(n)$ is equal to
\[
 H^*(\Fl(n),\ZZ)=\bigoplus_{w\in S_n}\ZZ\sigma_w,
\]
where $\sigma_w\in H^{2\ell(w)}(\Fl(n))$.
\end{prop}

This explains our choice of this twist: $\ell(w_0w)=n(n-1)/2-\ell(w)$, so $\codim X_{w_0w}=\ell(w)$, and the class $[X_{w_0w}]$ has degree $2\ell(w)$.

The previous proposition allows us to compute the Poincar\'e polynomial of $\Fl(n)$:

\begin{xca}
 Show that 
\[
 P_q(\Fl(n))=\frac{(1-q)(1-q^2)\dots(1-q^n)}{(1-q)^n}.
\]
\end{xca}

\begin{hint} The proof is similar to the proof of Corollary~\ref{cor:poincarepolynomial}: suppose that $q$ is a power of a prime number and count the number of points of a flag variety $\Fl(n,\FF_q)$ over the finite field $\FF_q$.
\end{hint}

As in the case of Grassmannians, let us introduce the \emph{dual Schubert varieties}, related to the dual flag $V_\bullet'$, where $V'_i=\langle e_{n+1-i},\dots,e_n\rangle$. Let
\[
 \Omega'_w=\{U_\bullet\in\Fl(V)\mid \dim (U_p\cap V_q')=r_{w_0w}(p,q), 1\leq p,q\leq n\},
\]
and let $X_w'=\overline{\Omega_w'}$. Again, $\Omega_w'$ is an affine space, but now its \emph{codimension}, not the dimension, is equal to $\ell(w)$. Every flag $U_\bullet\in\Omega_w'$  corresponds to a unique matrix whose $(i,w(i))$-th entries are equal to 1, the coefficients \emph{below} or to the \emph{left} of $1$'s are equal to zero, and all the remaining coefficients can be arbitrary.

\begin{example} Let $w=(25413)$. Then each element of $\Omega_{w}'$ corresponds to a uniquely determined matrix of the form
\[
 \begin{pmatrix}
  0 & 1 & * & * & *\\
  0 & 0 & 0 & 0 & 1\\
  0 & 0 & 0 & 1 & 0\\
  1 & 0 & * & 0 & 0\\
  0 & 0 & 1 & 0 & 0
 \end{pmatrix}
\]
\end{example}

From the transitivity of the action of $\GL(V)$ on flags we conclude that $X_w'$ and $X_{w_0w}$ have the same dual fundamental class $\sigma_w$.

We continue to follow the same program as in the case of Grassmannians by stating the duality result.

\begin{prop} Let $v,w\in S_n$, and $\ell(v)=\ell(w)$. Then
\[
 \sigma_v\cdot \sigma_{w_0w}=\delta_{v,w}.
\]
\end{prop}

\begin{xca} Prove this proposition, using the description of $\Omega'_{w_0v}$ and $\Omega_{w}$ given above.
\end{xca}

Structure constants of the ring $H^*(\Fl(V),\ZZ)$ are the coefficients $c_{wv}^u$ of decompositions
\[
 \sigma_w\cdot \sigma_v=\sum c_{wv}^u \sigma_u.
\]
(they are sometimes called \emph{generalized Littlewood--Richardson coefficients}).

Similarly to Theorem~\ref{thm:LR} for Grassmannians, Kleiman's transversality theorem implies their nonnegativity by means of the same geometric argument. One would be interested in a \emph{combinatorial} proof of their nonnegativity, analogous to the Littlewood--Richardson problem: how to describe sets of cardinalities $c_{wv}^u$? What do such sets index? This problem is open. One of the recent attempts to solve it is given in the unpublished preprint \cite{CoskunPreprint} by Izzet Co\c{s}kun; it uses the so-called \emph{Mondrian tableaux\footnote{Piet Mondrian was a Dutch artist, known for his abstract compositions of lines and colored rectangles; the combinatorial objects introduced by Co\c{s}kun for the study of Schubert varieties resemble Mondrian's paintings.}}.

For Grassmannians we had the Pieri rule which allowed us to multiply Schubert classes by some special classes. A similar formula holds for flag varieties, but instead of special classes it involves Schubert divisors, i.e. Schubert varieties of codimension 1. There are $n-1$ of them; they correspond to simple transpositions $s_1,\dots,s_{n-1}$. Recall that the simple transposition $s_i\in S_n$ exchanges $i$ with $i+1$ and leaves all the remaining elements fixed. We will also need arbitrary transpositions; denote a transposition exchanging $j$ with $k$ by $t_{ij}$.

\begin{theorem}[Chevalley--Monk formula] For each permutation $w\in S_n$ and each $i<n$,
\[
 \sigma_w\cdot\sigma_{s_i}=\sum_{j\leq i<k,\ell(wt_{ij})=\ell(w)+1} \sigma_{wt_{jk}},
\]
where the sum is taken over all transpositions $t_{jk}$ which increase the length of $w$ by $1$, and $j\leq i<k$.
\end{theorem}

We will not prove this theorem here; the reader may consider it as a nontrivial exercise or find its proof, for instance, in~\cite[Sec.~3.6.3]{Manivel98}.

\subsection{Fundamental example: $\Fl(3)$} 

Let $n=3$. A flag of vector subspaces in $\CC^3$ can be viewed as a flag of \emph{projective} subspaces in $\PP^2$, i.e., a pair $(p,\ell)$ consisting of a point and a line, such that $p\in \ell$. Let $(p_0,\ell_0)$ be the projectivization of the standard flag $V_\bullet$, i.e., $p_0=[\langle e_1\rangle]$ and $\ell_0=[\langle e_1,e_2\rangle]$. Here we list all Schubert varieties in the case of $\Fl(3)$. There are $3!=6$ of them. For each $w\in S_3$, we draw the standard flag (we will also call it \emph{fixed}) by a solid line and a black dot, and a generic element $(p,\ell)\in X_w$ (sometimes referred to as ``the moving flag'') by a dotted line and a white dot. For each $w$ we compute the permutation $w_0w$; the corresponding Schubert class is $[X_w]=\sigma_{w_0w}$.

\begin{itemize}
 \item $w=(321)$. This is the generic situation: there are no relations on the fixed and the moving flag, $X_{(321)}=\Fl(3)$. The corresponding Schubert class is $\sigma_{id}=1\in H^*(\Fl(3))$.

\item $w=(312)$. In this case $p_0\in \ell$. In the language of vector spaces this would mean that $U_2\supset V_2$, and $U_1$ can be arbitrary. $\dim X_{(312)}=2$. The twisted permutation $w_0w=(213)=s_1$ is the first simple transposition.

\item $w=(231)$: this is the second two-dimensional Schubert variety (or a \emph{Schubert divisor}). The condition defining it is $p\in \ell_0$, and $w_0w=(132)=s_2$ is the second simple transposition.

\item $w=(132)$: in this case the points $p_0=p$ collide. $w_0w=(231)=s_1s_2$. The set of flags $(p,\ell)$ such that $p=p_0$ forms a $B$-stable curve in the flag variety isomorphic to $\PP^1$.

\item $w=(213)$: this condition says that $\ell=\ell_0$. This is the second $B$-stable curve, also isomorphic to $\PP^1$; its permutation is $w_0w=(312)=s_2s_1$.

\item $w=(123)$: this is the unique zero-dimensional Schubert variety, given by the conditions $p=p_0$ and $\ell=\ell_0$. The twisted permutation $w_0w=(312)=s_1s_2s_1=s_2s_1s_2$ is the longest one, and the corresponding Schubert class $\sigma_{w_0}$ is the class of a point.
\end{itemize}

\begin{figure}[h!]
 \begin{tikzpicture}
  \draw[ultra thick] (-2,0)--(2,0);
  \draw[fill] (0,0) circle [radius=0.1];
  \draw[dotted, ultra thick] (2,1.5)--(-2,-0.5);
  \draw[fill=white] (1,1) circle [radius=0.1];
  \node at (0,-1) {$[X_{(321)}]=\sigma_{id}$};
 \end{tikzpicture}
\qquad 
 \begin{tikzpicture}
  \draw[ultra thick] (-2,0)--(2,0);
  \draw[fill] (0,0) circle [radius=0.1];
  \draw[dotted, ultra thick] (1.5,1.5)--(-0.5,-0.5);
  \draw[fill=white] (1,1) circle [radius=0.1];
  \node at (0,-1) {$[X_{(312)}]=\sigma_{s_1}$};
 \end{tikzpicture}\\
\begin{tikzpicture}
  \draw[ultra thick] (-2,0)--(2,0);
  \draw[fill] (0,0) circle [radius=0.1];
  \draw[dotted, ultra thick] (2,1.5)--(-2,-0.5);
  \draw[fill=white] (-1,0) circle [radius=0.1];
  \node at (0,-1) {$[X_{(231)}]=\sigma_{s_2}$};
 \end{tikzpicture}\qquad
\begin{tikzpicture}
  \draw[ultra thick] (-2,0)--(2,0);
  \draw[dotted, ultra thick] (2,2)--(-0.5,-0.5);
  \draw[fill=white] (0.05,0) circle [radius=0.1];
  \draw[fill] (-0.05,0) circle [radius=0.1];
  \node at (0,-1) {$[X_{(312)}]=\sigma_{s_1s_2}$};
 \end{tikzpicture}
\medskip
\begin{tikzpicture}
  \draw[ultra thick] (-2,0)--(2,0);
  \draw[dotted, ultra thick] (-2,0.07)--(2,0.07);
  \draw[fill=white] (-1,0) circle [radius=0.1];
  \draw[fill] (0.0,0) circle [radius=0.1];
  \node at (0,-1) {$[X_{(213)}]=\sigma_{s_2s_1}$};
  \node at (0,1) {~};
 \end{tikzpicture}
\qquad
\begin{tikzpicture}
  \draw[ultra thick] (-2,0)--(2,0);
  \draw[dotted, ultra thick] (-2,0.07)--(2,0.07);
  \draw[fill=white] (0.05,0) circle [radius=0.1];
  \draw[fill] (-0.05,0) circle [radius=0.1];
  \node at (0,-1) {$[X_{(321)}]=\sigma_{s_1s_2s_1}$};
  \node at (0,1) {~};
 \end{tikzpicture}
\caption{Schubert varieties in $\Fl(3)$}\label{fig:flags3}
\end{figure}

Note that the Bruhat order can be seen on these pictures: $v\leq w$ if and only if a moving flag corresponding to $w$ can be degenerated to a moving flag corresponding to $v$, i.e., if $\Omega_v\subset\overline{\Omega_w}$.

These pictures allow us to compute the products of certain Schubert classes. 

\begin{example}\label{ex:square} Let us compute $\sigma_{s_1}^2$. This means that we have two fixed flags $(p_0,\ell_0)$ and $(\tilde p_0,\tilde\ell_0)$ in a general position with respect to each other, and we are looking for moving flags $(p,\ell)$ satisfying the conditions for $\sigma_{s_1}$, namely, $p_0\in\ell$ and $\tilde p_0\in \ell$. Each of these Schubert varieties, which we denote by $X^{(312)}$ and $\widetilde X^{(312)}$ is of codimension 1, so their intersection has the expected codimension 2. Indeed, such flags are given by the condition $\ell=\langle p_0,\tilde p_0\rangle$. This means that the position of the line $\ell$ is prescribed. But this is exactly the condition defining the Schubert class $\sigma_{s_2s_1}$ (cf.~Figure~\ref{fig:schubproduct}). 

It remains to show that the intersection of $X_{(312)}$ and $\widetilde X_{(312)}$ is transversal. Informally this can be seen as follows: the tangent vectors to each of the Schubert varieties correspond to moving flags $(p',\ell')$ which are ``close'' to the flag $(p,\ell)$ and satisfy the conditions $p_0\in\ell'$ and $\tilde p_0\in\ell'$ respectively. So the tangent space to each of these Schubert varieties at $(p,\ell)$ is two-dimensional, with natural coordinates corresponding to infinitesimal shifts of $p$ along $\ell$ and infinitesimal rotations of $\ell$ along $p_0$ and $\tilde p_0$, respectively. The intersection of these two subspaces is a line corresponging to the shifts of $p$ along $\ell$, hence one-dimensional. So $\sigma_{s_1}^2=\sigma_{s_2s_1}$.
\begin{figure}
\begin{tikzpicture}
  \draw[ultra thick] (-2,0)--(2,0);
  \draw[fill] (0,0) circle [radius=0.1];
  \draw[dotted, ultra thick] (2,2)--(-1,-1);
  \draw[ultra thick] (1.5,-1)--(1.5,2);
  \draw[fill] (1.5,1.5) circle [radius=0.1];
  \draw[fill=white] (1,1) circle [radius=0.1];
  \node[below] at (0,0) {$p_0$};
  \node[right] at (1.5,1.5) {$\tilde p_0$};
  \node[below] at (1,1) {$p$};
  \node[right] at (1.5,-0.5) {$\tilde\ell_0$};
  \node[below] at (-1.5,0) {$\ell_0$};
  \node[below] at (-0.5,-0.5) {$\ell$};
 \end{tikzpicture} 
\caption{$\sigma_{s_1}^2=\sigma_{s_2s_1}$}\label{fig:schubproduct}
\end{figure}
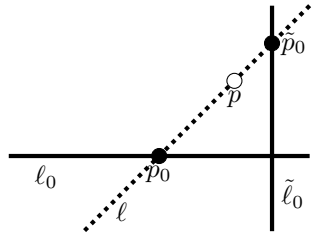
\end{example}

\begin{xca}
 Show in a similar way that $\sigma_{s_2}^2=\sigma_{s_1s_2}$ and $\sigma_{s_1}\sigma_{s_2}=\sigma_{s_1s_2}+\sigma_{s_2s_1}$.
\end{xca}

\subsection{Borel presentation and Schubert polynomials}

There is another presentation of the cohomology ring of a full flag variety, due to Armand Borel~\cite{Borel53}. We will give its construction without proof; details can be found in \cite[Sec.~3.6.4]{Manivel98}.

Let $\Fl(n)$ be a full flag variety. Consider $n$ tautological vector bundles $\cV_1,\dots,\cV_n$ of ranks $1,\dots,n$. By definition, the fiber of $\cV_i$ over any point $V_\bullet$ is $V_i$. Since $\cV_{i-1}$ is a subbundle of $\cV_i$, we can take the quotient line bundle $\ccL_i=\cV_i/\cV_{i-1}$. 

\begin{theorem}\label{thm:borel} Consider a morphism from the polynomial ring $\ZZ[x_1,\dots,x_n]$ in $n$ variables into $H^*(\Fl(n),\ZZ)$, taking each variable $x_i$ to the negative first Chern class of $\ccL_i$:
\[
 \varphi\colon \ZZ[x_1,\dots,x_n]\to H^*(\Fl(n),\ZZ),\qquad x_i\mapsto -c_1(\ccL_i).
\]
Then $\varphi$ is a surjective morphism of graded rings, and $\Ker\varphi=I$ is the ideal generated by all symmetric polynomials in $x_1,\dots,x_n$ with zero constant term.
\end{theorem}

This presentation gives rise to a natural question: if $H^*(\Fl(n))$ is the quotient of a polynomial ring, how to find polynomials in $\ZZ[x_1,\dots,x_n]$ representing Schubert classes? Of course, a preimage of $\sigma_w$ in $\ZZ[x_1,\dots,x_n]$ is not uniquely defined: this is a coset modulo the ideal $I$. Let us pick a ``lift'' $M$ of $H^*(\Fl(n))$ into $\ZZ[x_1,\dots,x_n]$ as follows. For two monomials $x^I=x_1^{i_1}\dots x_n^{i_n}$ and $x^J=x_1^{j_1}\dots x_n^{j_n}$ we will say that $x^I$ is \emph{dominated} by $x^J$ iff $i_\alpha\leq j_\alpha$ for each $\alpha\in\{1,\dots,n\}$. 

Let
\[
 M=\langle x_1^{i_1}x_2^{i_2}\dots x_n^{i_n}\mid 0\leq i_k\leq n-k\rangle_\ZZ
\]
be the $\ZZ$-span of all monomials dominated by the ``staircase monomial'' $x_1^{n-1}x_2^{n-2}\dots x_{n-1}$. In particular, all monomials in $M$ do not depend on $x_n$. Then $M$ is a free abelian subgroup of rank $n!$, and 
\[
 \ZZ[x_1,\dots,x_n]=I\oplus M
\]
as abelian groups. So for each element $y\in H^*(\Fl(n))$ there exists a unique $x\in M$ such that $\varphi(x)=y$.

\begin{definition} Let $\fS_w(x_1,\dots,x_{n-1})$ be a polynomial from $M$ such that $\varphi(\fS_w)=\sigma_w$. Then $\fS_w$ is called the \emph{Schubert polynomial} corresponding to $w$.
\end{definition}

\begin{example} $\fS_{id}=1$, and $\fS_{w_0}=x_1^{n-1}x_2^{n-2}\dots x_{n-1}$.
\end{example}

This definition may seem unnatural at the first glance, since it depends on the choice of $M$. However, Schubert polynomials defined in such a way satisfy the following \emph{stability property}.

Consider a natural embedding $\CC^n\hookrightarrow\CC^{n+1}$ whose image consists of vectors whose last coordinate is zero. It defines an embedding of full flag varieties $\iota_n\colon \Fl(n)\to\Fl(n+1)$. This map defines a surjective map of cohomology rings: $\iota_n^*\colon H^*(\Fl(n+1))\to H^*(\Fl(n))$. 

One can easily see what happens with Schubert classes under this map. Let $w\in S_{n}$. Denote by $w\times 1\in S_{n+1}$  the image of $w$ under the natural embedding $S_n\hookrightarrow S_{n+1}$. Then
\[
 \iota^*(\sigma_v)=\begin{cases} \sigma_w & \text{ if } v=w\times 1,\\ 0 & \text{ otherwise}.
                                       \end{cases}
\]

Let $M_n\subset \ZZ[x_1,\dots,x_n]$ and $M_{n+1}\subset \ZZ[x_1,\dots,x_{n+1}]$ be the free abelian subgroups spanned by monomials dominated by the corresponding staircase monomials $x_1^{n-1}x_2^{n-2}\dots x_{n-1}$ and $x_1^n x_2^{n-1}\dots x_n$ (note that the monomials in $M_{n}$ and $M_{n+1}$ do \emph{not} depend on $x_n$ and $x_{n+1}$, respectively). There is a surjective map
\[
 \mu_n\colon M_{n+1}\to M_n,
\]
\[
  \mu_n(x_1^{i_1}\dots x_n^{i_n})=\begin{cases} x_1^{i_1}\dots x_n^{i_n}, & i_k<n-k\text{ for each }k\leq n;\\
                                 0 &\text{otherwise}.
                                \end{cases}
\]
(In particular, every monomial containing $x_{n}$ is always mapped to zero). The diagram 
\[
 \xymatrix{ 
M_{n+1}\ar[r]\ar[d]_{\mu_n} & H^*(\Fl(n+1))\ar[d]^{\iota^*_n}\\
M_n \ar[r] & H^*(\Fl(n))
}
\]
is commutative.

We can consider the colimit $\lim\limits_{\leftarrow} M_{n}=\ZZ[x_1,x_2,\dots]$. The Schubert polynomial $\fS_w(x_1,x_2,\dots)\in\ZZ[x_1,x_2,\dots]$ is then the unique polynomial which is mapped to $\sigma_w$ for $n$ sufficiently large.

\subsection{Divided difference operators, pipe dreams and the Fomin--Kirillov theorem}\label{ssec:pipedreams}
The method of computation of Schubert polynomials (as well as the definition of this notion itself) was given by Lascoux and Sch\"utzenberger \cite{LascouxSchutzenberger82}. Essentially the same construction appeared several years before in the paper   \cite{BernsteinGelfandGelfand73} by J.~Bernstein, I.~Gelfand and S.~Gelfand. It is as follows.

Consider the ring $\ZZ[x_1,\dots,x_n]$. Define the \emph{divided difference operators} $\partial_1,\dots,\partial_{n-1}$:
\[
 \partial_i(f)=\frac{f(x_1,\dots,x_n)-f(x_1,\dots,x_{i+1},x_i,\dots,x_n)}{x_i-x_{i+1}}.
\]

\begin{xca}\label{xca:welldef} Show that: 
\begin{enumerate}
             \item $\partial_i$ takes a polynomial into a polynomial;
\item $\partial_i^2=0$;
\item $\partial_i\partial_j=\partial_j\partial_i$ for $|i-j|>1$;
\item $\partial_i\partial_{i+1}\partial_i=\partial_{i+1}\partial_{i}\partial_{i+1}$.
            \end{enumerate}
\end{xca}

Let $w\in S_n$ be a permutation. Let us multiply it by $w_0$ from the left and  consider a presentation of the resulting permutation as the product of simple transpositions:
\[
 w_0w=s_{i_1}s_{i_2}\dots s_{i_r}.
\]
(some of the $i_k$'s can be equal to each other).  Such a presentation is called a \emph{reduced decomposition} if the number of factors is the smallest possible, i.e., equal to the length $\ell=\ell(w_0 w)$ of the permutation $w_0w$.

\begin{theorem}[\cite{LascouxSchutzenberger82},  \cite{BernsteinGelfandGelfand73}]\label{thm:divdiff} For such a $w\in S_n$,
\[
 \fS_{w}(x_1,\dots,x_{n-1})=\partial_{i_\ell}\dots\partial_{i_2}\partial_{i_1}(x_1^{n-1}x_2^{n-2}\dots x_{n-1}),
\]
where $w_0 w=s_{i_1}s_{i_2}\dots s_{i_\ell}$ is a reduced decomposition of $w_0 w$.
\end{theorem}

\begin{remark} $\fS_w$ depends only on the permutation $w$ and does not depend on the choice of its reduced decomposition. Indeed, one can pass from any reduced decomposition of $w_0w$ to any other using the relations $s_is_j=s_js_i$ for $|i-j|>1$ and $s_is_{i+1}s_i=s_{i+1}s_is_{i+1}$ (the proof of this well-known fact can be found, for instance, in \cite[Sec.~2.1]{Manivel98} or \cite[Chapter~1]{Humphreys90}). Exercise~\ref{xca:welldef} states that the divided difference operators satisfy these relations as well.
\end{remark}

\begin{xca}\label{xca:schubert_s3} Compute the Schubert polynomials for all six permutations in $S_3$.
\end{xca}

\begin{hint} The answer is as follows:
\begin{eqnarray*}
  \fS_{id}=1; & \fS_{s_1}=x; & \fS_{s_2}=x+y;\\
  \fS_{s_1s_2}=xy; & \fS_{s_2s_1}=x^2; & \fS_{s_1s_2s_1}=x^2y.
\end{eqnarray*}
\end{hint}

\begin{xca} Show that for $s_i\in S_n$, the Schubert polynomial equals $\fS_{s_i}=x_1+\dots+x_i$.
\end{xca}

Note that all the coefficients of Schubert polynomials in these examples are nonnegative. It turns out that this is always the case. From Theorem~\ref{thm:divdiff} this is absolutely unclear, since the divided difference operator involves subtractions; however, after all these subtractions and divisions we always get a polynomial with positive coefficients. This was shown independently by Fomin and Stanley \cite{FominStanley94} and Billey, Jockush and Stanley \cite{BilleyJockuschStanley93} (the original conjecture is due to Stanley, and that is why his name is on two ``independent'' papers). 

In \cite{BergeronBilley93} and \cite{FominKirillov96}, a \emph{manifestly positive} rule for computing Schubert polynomials was proposed. We will describe this rule now. For this we will need to define  combinatorial objects called \emph{pipe dreams}, or \emph{rc-graphs}.

Consider an $(n\times n)$-square divided into $(1\times 1)$-squares. We will fill the small squares by two types of elements, ``crosses'' $\textcross$ and ``elbow joints'' $\textelbow$. First, let us put elbow joints in all squares on the antidiagonal and below it. Above the antidiagonal, let us put elements of these two types in an arbitrary way. We will get something like Figure~\ref{fig:pipedreamsample}, left.
\begin{figure}[h!]
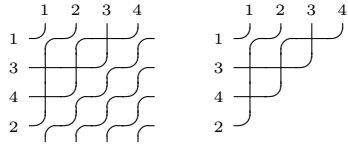

\begin{rcgraph}
\begin{shortstretchpipedream}{4}
 &     \perm1{}&\perm2{}&\perm3{}&\perm4{}\\
\petit1 &   \jr  &   \jr  &   \+   & \jr\\
\petit3 &  \+   &   \+   &   \jr  & \jr\\
\petit4 &  \+   &   \jr  &   \jr  & \jr\\
\petit2 &  \jr  &   \jr  &   \jr  & \jr\\
\end{shortstretchpipedream}
\qquad
\begin{shortstretchpipedream}{4}
 &     \perm1{}&\perm2{}&\perm3{}&\perm4{}\\
\petit1 &   \jr &   \jr  &   \+   & \je\\
\petit3 &  \+   &   \+   &   \je  & \\
\petit4 &  \+   &   \je  &        & \\
\petit2 &  \je  &        &        &\\
\end{shortstretchpipedream}
\end{rcgraph}
\caption{A pipe dream}\label{fig:pipedreamsample}
\end{figure}
In this picture we see a configuration of four strands starting at the left edge of the square and ending on the top edge in a different order. Such a configuration is called a \emph{pipe dream}. Let us put numbers $1,\dots,n$ on the top ends of the strands and put the same number on the left end of each strand. Then the reading of the numbers on the left edge gives us a permutation (in the example on Figure~\ref{fig:pipedreamsample} this permutation is equal to $(1342)$). Let us denote the permutation corresponding to a pipe dream $P$ by $\pi(P)$. The part below the antidiagonal plays no essential role, so further we will just omit it (see Figure~\ref{fig:pipedreamsample}, right).

A pipe dream is said to be \emph{reduced} if each pair of strands intersects at most once. The pipe dream on Figure~\ref{fig:pipedreamsample} is not reduced, since the strands $3$ and $4$ intersect twice. We will consider only reduced pipe dreams.

It is clear that for a permutation $w$ there can be more than one reduced pipe dream $P$ with $\pi(P)=w$. The first example is given by $w=(132)=s_2$: it corresponds to two such pipe dreams, shown on Figure~\ref{fig:s2pipedream} below.
\begin{figure}[h!]
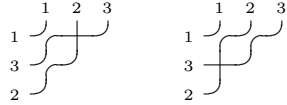

\begin{rcgraph}
 \begin{shortstretchpipedream}{3}
 &     \perm1{}&\perm2{}&\perm3{}\\
\petit1 &   \jr &   \+  & \je\\
\petit3 &   \jr   &   \je  \\
\petit2 &   \je  &        \\
\end{shortstretchpipedream}
\qquad
 \begin{shortstretchpipedream}{3}
 &     \perm1{}&\perm2{}&\perm3{}\\
\petit1 &   \jr &   \jr   & \je\\
\petit3 &   \+   &   \je  & \\
\petit2 &   \je  &        & \\
\end{shortstretchpipedream}
\end{rcgraph}
\caption{Two reduced pipe dreams of $w=(132)$}\label{fig:s2pipedream}
\end{figure}

\begin{xca}
 Let $P$ be a reduced pipe dream, $\pi(P)=w$. Show that the number of crosses in $P$ equals $\ell(w)$.
\end{xca}

Let $P$ be an arbitrary pipe dream with $n$ strands. Denote by $d(P)$ the monomial $x_1^{i_1}x_2^{i_2}\dots x_{n-1}^{i_{n-1}}$, where $i_k$ is the number of crosses in the $k$-th row (note that the $n$-th row never contains crosses). The monomials corresponding to pipe dreams on Figure~\ref{fig:s2pipedream} correspond to monomials $x_1$ and $x_2$, respectively.

The following theorem, usually called the Fomin--Kirillov theorem, expresses the Schubert polynomial of a permutation as a sum of monomials corresponding to pipe dreams.
\begin{theorem}[\cite{BergeronBilley93}, \cite{FominKirillov96}]\label{thm:fominkirillov} Let $w\in S_n$. The Schubert polynomial of $w$ is equal to
\[
 \fS_w=\sum_{\pi(P)=w} d(P),
\]
where the sum is taken over all reduced pipe dreams corresponding to $w$. 
\end{theorem}

This theorem implies positivity of coefficients of Schubert polynomials.

\begin{example} Let $w=(1432)$. Then there are five reduced pipe dreams corresponding to $w$, see Figure~\ref{fig:1432pipedream}. We conclude that
\[
 \fS_{(1432)}(x,y,z)=x^2y+xy^2+x^2z+xyz+y^2z.
\]
\end{example}

\begin{figure}[h!]
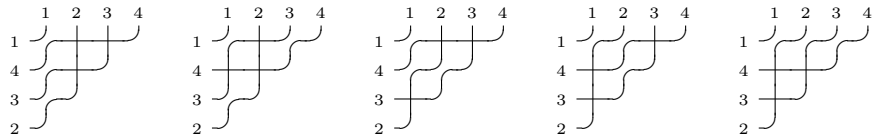

\begin{rcgraph}
\begin{shortstretchpipedream}{4}
       &\perm1{}&\perm2{}&\perm3{}&\perm4{}\\
\petit1&   \jr  &   \+   &   \+   &  \je   \\
\petit4&   \jr  &   \+   &   \je  &\\
\petit3&   \jr  &   \je  &        &\\
\petit2&   \je  &        &        &\\
\end{shortstretchpipedream}
\quad
\begin{shortstretchpipedream}{4}
       &\perm1{}&\perm2{}&\perm3{}&\perm4{}\\
\petit1&   \jr  &   \+   &   \jr  &  \je   \\
\petit4&   \+   &   \+   &   \je  &\\
\petit3&   \jr  &   \je  &        &\\
\petit2&   \je  &        &        &\\
\end{shortstretchpipedream}
\quad
\begin{shortstretchpipedream}{4}
       &\perm1{}&\perm2{}&\perm3{}&\perm4{}\\
\petit1&   \jr  &   \+   &   \+    &  \je   \\
\petit4&   \jr &    \jr  &   \je  &\\
\petit3&   \+   &   \je  &        &\\
\petit2&   \je  &        &        &\\
\end{shortstretchpipedream}
\quad
\begin{shortstretchpipedream}{4}
       &\perm1{}&\perm2{}&\perm3{}&\perm4{}\\
\petit1&   \jr  &   \jr  &   \+    &  \je   \\
\petit4&   \+   &   \jr  &   \je  &\\
\petit3&   \+   &   \je  &        &\\
\petit2&   \je  &        &        &\\
\end{shortstretchpipedream}
\quad
\begin{shortstretchpipedream}{4}
       &\perm1{}&\perm2{}&\perm3{}&\perm4{}\\
\petit1&   \jr  &   \jr  &   \jr  &  \je   \\
\petit4&   \+   &   \+   &   \je  &\\
\petit3&   \+   &   \je  &        &\\
\petit2&   \je  &        &        &\\
\end{shortstretchpipedream}
\end{rcgraph}
\caption{Five reduced pipe dreams of $w=(1432)$}\label{fig:1432pipedream}
\end{figure}

\begin{xca} Draw all pipe dreams for all remaining permutations from $S_3$ and compare the result with Exercise~\ref{xca:schubert_s3}.
\end{xca}

\section{Toric varieties}\label{sec:toric}

In the remaining part of the paper we will describe a new approach to Schubert calculus on full flag varieties. We will mostly follow the paper \cite{KiritchenkoSmirnovTimorin12}. In this approach we generalize some notions from the theory of toric varieties and see toric methods working with some modifications in a non-toric case.

In this section we speak about toric varieties and lattice polytopes. In Subsection~\ref{ssec:toricdef} we recall some basic facts about polarized projective toric varieties (this is the only class of toric varieties we will need). This is by no means an introduction into theory of toric varieties; a very nice introduction can be found in Danilov's survey \cite{Danilov78} or Fulton's book \cite{Fulton93}, or in the recent book by Cox, Little, and Schenck \cite{CoxLittleSchenck11}. In the second part of the latter book the authors give an overview of the results of Khovanskii and Pukhlikov on the toric Riemann--Roch theorem; these results are used in the proof of the Khovanskii--Pukhlikov theorem on the cohomology ring of  a smooth toric variety. We discuss this theorem in Subsection~\ref{ssec:pkh}; it will play a crucial role for our construction.

\subsection{Definition, examples and the first properties}\label{ssec:toricdef}


Recall that a normal algebraic variety is called \emph{toric} if it is equipped with an action of an algebraic torus $(\CC^*)^n$, and this action has an open dense orbit.

Consider a polytope $P\subset \RR^n$ with integer vertices. We suppose that $P$ is not contained in a hyperplane. $P$ is called a \emph{lattice polytope} if all vertices of $P$ belong to $\ZZ^n\subset \RR^n$.

Let $A=P\cap\ZZ^n=\{m_{0},\dots,m_N\}$ be the set of all lattice points in $P$, where $N=|A|-1$. Consider a projective space $\PP^N$ with homogeneous coordinates $(x_0:\cdots:x_N)$ indexed by points from $A$. 
For a point $m_i=(m_{i1},\dots,m_{in})\in A$ and a point of the torus $(t_1,\dots,t_n)\in T$, set $t^{m_i}:=t_1^{m_{i1}}\dots t_n^{m_{in}}$.
Now consider the embedding $\Phi_A\colon T\hookrightarrow\PP^N$, defined as follows:
\[
 \Phi_A\colon t\mapsto (t^{m_0}:\dots:t^{m_N}).
\]

\begin{xca} Prove that this map is an embedding.
\end{xca}

Let $X=\overline{\Phi_A(T)}$ be the closure of the image of this map. $X$ is a polarized projective toric variety. The word ``polarized'' means that it comes with an embedding into a projective space, or, equivalently, that we fix a very ample divisor on $X$.

\begin{xca} Show that there is a dimension-preserving bijection between $T$-orbits on $X$ and faces of $P$. The open orbit corresponds to the polytope $P$ itself.
\end{xca}

\begin{theorem}[{\cite[Chapter 2]{CoxLittleSchenck11}}] Any polarized projective toric variety can be obtained in such a way from a certain lattice polytope $P$. Two varieties are isomorphic if the corresponding polytopes have the same normal fan.
\end{theorem}

In the following three examples the torus is two-dimensional, and the polytopes are just polygons.

\begin{example}
 Let $P$ be a triangle with vertices $(0,0)$, $(1,0)$, and $(0,1)$. The torus orbit is formed by the points $(1:t_1:t_2)\in\PP^2$, and its closure is the whole $\PP^2$.
\end{example}

\begin{example} In a similar way, consider a right isosceles triangle with vertices $(0,0)$, $(k,0)$, and  $(0,k)$. It defines the following embedding of $(\CC^*)^2$ into  $\PP^2$:
\[
 (t_1,t_2)\mapsto (\cdots:t_1^i t_2^j:\cdots),\text{ where } i+j\leq k.
\]
Its closure is the image of the $k$-th Veronese embedding $v_k\colon \PP^2\hookrightarrow \PP^{k(k+1)/2-1}$.
\end{example}

Note that in these two examples we get two different embeddings of the same variety, and the corresponding polytopes have the same normal fan. 

\begin{example} Let $P$ be a unit square. The embedding $T=(\CC^*)^2\hookrightarrow \PP^3$ is then given by
\[
 (t_1,t_2)\mapsto (1:t_1:t_2:t_1t_2).
\]
The closure of its image is given by the relation $x_0x_3=x_1x_2$. It is isomorphic to $\PP^1\times \PP^1$ embedded by Segre into $\PP^3$. 
\end{example}

More details on Segre and Veronese embeddings can be found in \cite{Reid88} or \cite{Harris92}.

Recall that a polytope $P\subset \RR^n$ is said to be \emph{simple} if it has exactly $n$ edges meeting in each vertex. (I.e., a cube is simple, while an octahedron is not). Let $P\subset\ZZ^n$ be a simple lattice polytope. For each of its vertices $v$, consider the set of edges adjacent to this vertex and for each edge take the \emph{primitive vector}, i.e., the vector joining $v$ with the nearest lattice point on its edge. For each $v$ we get a set of lattice vectors. The polytope $P$ is called \emph{integrally simple} if for each $v$ such a set of vectors forms a basis of the lattice $\ZZ^n$.

\begin{example} Let $k>0$. A triangle with vertices $(0,0)$, $(k,0)$ and $(1,0)$ is integrally simple iff $k=1$.  The corresponding toric variety is the \emph{weighted projective plane} $\PP(1,1,k)$.
\end{example}

The following theorem gives a criterion for smoothness of a toric variety.

\begin{theorem}\cite[Theorem~2.4.3]{CoxLittleSchenck11} A projective toric variety $X$ is smooth iff the corresponding lattice polytope is integrally simple. 
\end{theorem}

\subsection{The Khovanskii--Pukhlikov ring}\label{ssec:pkh} Our next goal is to describe the integral cohomology ring $H^*(X,\ZZ)$ of a smooth toric variety. This was first done in Danilov's survey~\cite[Sec.~10]{Danilov78}. Danilov speaks about the Chow ring $A^*(X)$ rather than about the cohomology ring, but for smooth toric varieties over $\CC$ these rings are known to be isomorphic (loc.~cit., Theorem~10.8).

We will give a description of $H^*(X,\ZZ)$ which implicitly appeared in the paper by A.~G.~Khovanskii and A.~V.~Pukhlikov \cite{KhovanskiiPukhlikov93} and was made explicit by K.~Kaveh~\cite{Kaveh11}. We begin with a construction which produces a finite-dimensional commutative ring starting from a lattice polytope. To do this, let us first recall some definitions.

Let $P\subset \RR^n$ be a polytope not contained in a hyperplane, and let $h=a_0+a_1x_1+\dots+a_nx_n$ be an affine function. The hyperplane defined by this function is called a \emph{supporting hyperplane} if $h(x)\leq 0$ for each point $x\in P$ and the set $\{x\in P\mid h(x)=0\}$ is nonempty. The intersection of $P$ with a supporting hyperplane is called \emph{face}; faces of dimension $n-1$, 1, and 0 are called \emph{facets}, \emph{edges}, and \emph{vertices}, respectively. 

With each face $F$ we can associate the set of linear parts $(a_1,\dots,a_n)$ of all supporting hyperplanes corresponding to $F$. It is a closed strongly convex cone in $\RR^n$. It is called the \emph{normal cone to $P$ along $F$}. The set of all normal cones spans a complete fan, called the \emph{normal fan} of $P$.

We will say that two polytopes $P,Q\subset \RR^n$ are \emph{analogous} (notation: $P\sim Q$) if they have the same normal fan. The Minkowski sum $P+Q$ of two analogous polytopes is analogous to each of them. Polytopes also can be multiplied by nonnegative real numbers; $\lambda P$ is obtained from $P$ by dilation with the coefficient $\lambda$. Clearly, $\lambda P\sim P$. This means that the set of all polytopes analogous to a given polytope $P$ forms a semigroup with multiplication by positive numbers. Denote this semigroup by $S_P$.

\begin{xca} Show that $S_P$ has a cancellation property: if $P+R=Q+R$, then $P=Q$.
\end{xca}

\begin{example} The first two polygons on Figure~\ref{fig:normalfans} are analogous to each other, while the third one is not analogous to them. Their normal fans are depicted below.
\begin{figure}[h!]
\begin{tikzpicture}
\draw[ultra thick] (0,0)--(3,0)--(1,2)--(0,2)--cycle;
\draw[->] (0.5,0)--(0.5,-0.5);
\draw[->] (0.5,2)--(0.5,2.5);
\draw[->] (0,1)--(-0.5,1);
\draw[->] (2,1)--(2.3,1.3);
\end{tikzpicture}
\begin{tikzpicture}
\draw[ultra thick] (0,0)--(4,0)--(3,1)--(0,1)--cycle;
\draw[->] (1.5,0)--(1.5,-0.5);
\draw[->] (1.5,1)--(1.5,1.5);
\draw[->] (0,0.5)--(-0.5,0.5);
\draw[->] (3.5,0.5)--(3.8,.8);
\end{tikzpicture}
\begin{tikzpicture}
\draw[ultra thick] (3,2)--(0,2)--(2,0)--(3,0)--cycle;
\draw[->] (2.5,0)--(2.5,-0.5);
\draw[->] (2.5,2)--(2.5,2.5);
\draw[->] (3,1)--(3.5,1);
\draw[->] (1,1)--(0.7,0.7);
\end{tikzpicture}

\begin{tikzpicture}
\path[fill=gray] (-2,-2) rectangle (0,0);
\path[fill=gray!60] (-2,2) rectangle (0,0);
\path[fill=gray!30] (0,0)--(2,2)--(0,2)--cycle;
\path[fill=gray!70] (0,0)--(2,2)--(2,-2)--(0,-2)--cycle;
\draw[ultra thick] (0,-2.5)--(0,2.5);
\draw[ultra thick] (-2.5,0)--(0,0)--(2.5,2.5);
\end{tikzpicture}
\begin{tikzpicture}
\path[fill=gray] (2,2) rectangle (0,0);
\path[fill=gray!60] (2,-2) rectangle (0,0);
\path[fill=gray!30] (0,0)--(-2,-2)--(0,-2)--cycle;
\path[fill=gray!70] (0,0)--(-2,-2)--(-2,2)--(0,2)--cycle;
\draw[ultra thick] (0,2.5)--(0,-2.5);
\draw[ultra thick] (2.5,0)--(0,0)--(-2.5,-2.5);
\end{tikzpicture}
 \caption{Polytopes and their normal fans}\label{fig:normalfans}
\end{figure}
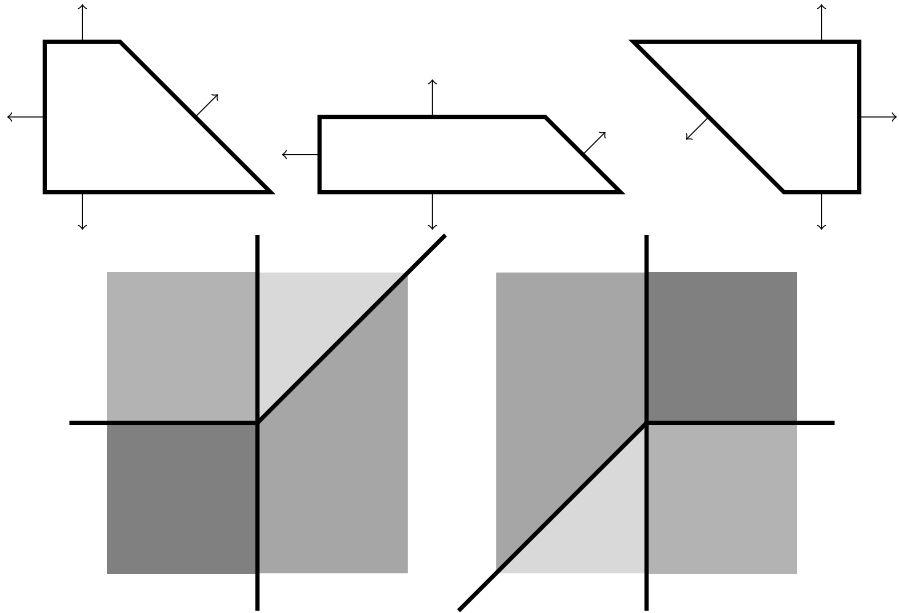

\end{example}

Consider the Grothendieck group of $S_P$ by adding formal differences of polytopes, with obvious equivalence relations. Denote this group by $V_P$; its elements are called \emph{virtual polytopes} analogous to $P$. Virtual polytopes can be multiplied by any real numbers, so $V_P$ is a vector space. It is clear that this space is finite-dimensional.

\begin{example} Let $P$ be simple. Then $V_P$ has a natural coordinate system, given by the \emph{support numbers}, i.e., the distances from the origin to the facets of $P$ (cf. Figure~\ref{fig:support}). The points of $V_P$ such that all its coordinates are positive correspond to the ``actual'' polytopes (i.e., elements of $S_P\subset V_P$) containing the origin. Thus, in this case $\dim V_P$ is equal to the number of facets of $P$.
\end{example}

\begin{figure}[h!]
\begin{tikzpicture}
  \draw[ultra thick] (0,0)--(9,0)--(3,6)--(0,6)--cycle;
  \draw[fill] (2,3) circle [radius=0.05];
  \draw (2,0)--(2,6);
  \draw (0,3)--(2,3)--(4,5);
  \node[above] at (1,3) {${h_1}$};
  \node[left] at (2,4.5) {${h_2}$};
  \node[left] at (2,1.5) {${h_4}$};
  \node[below right] at (3,4) {${h_3}$};
  \draw (1.7,0)--(1.7,0.3)--(2,0.3);
  \draw (1.7,6)--(1.7,5.7)--(2,5.7);
  \draw (0,3.3)--(0.3,3.3)--(0.3,3);
  \draw (3.8,4.8)--(3.6,5)--(3.8,5.2);
\end{tikzpicture}

\caption{Support numbers}\label{fig:support}
\end{figure}
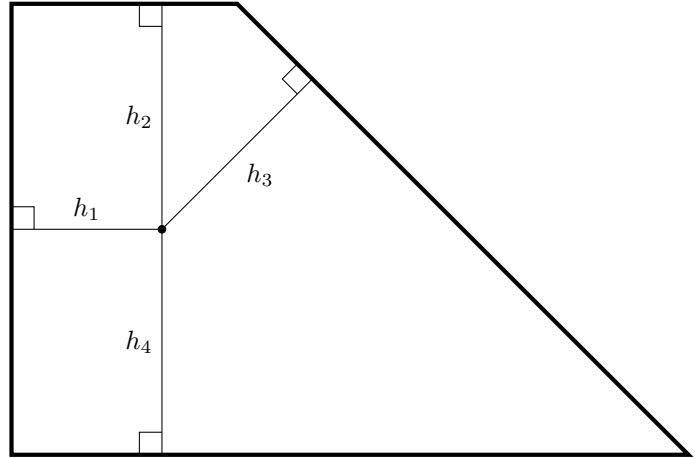

Note that for a nonsimple polytope $P$ there are relations on the support numbers (they cannot be changed independently from each other), so the space $V_P$ is a proper subspace in the vector space generated by support numbers.

Define the \emph{volume polynomial} $\vol_P\colon V_P\to \RR$ as follows. For each polytope $Q\in S_P$, let $\vol_P(Q)\in\RR$ be the volume of $Q$. This function can be extended to a unique homogeneous polynomial function of degree $n$ on $V_P$ (cf.~\cite{Kaveh11}).

\begin{definition}
Consider the (commutative) ring of all differential operators with integer coefficients $\Diff_\ZZ(V_P)$ on the space $V_P$. Let $\Ann(\vol_P)$ be the annihilator ideal of the volume polynomial $\vol_P$. The \emph{Khovanskii--Pukhlikov ring} of $P$ is the quotient of $\Diff_\ZZ(V_P)$ modulo this ideal:
\[
 R_P:=\Diff_\ZZ(V_P)/\Ann(\vol_P).
\]
\end{definition}

Since the polynomial $\vol_P$ is homogeneous, this ring inherits the grading from $\Diff_\ZZ(V_P)$. It is finite-dimensional, since any differential operator of degree greater than $n$ annihilates $\vol_P$. It also has a pairing: for two homogeneous differential operators $D_1,D_2$ such that $\deg D_1+\deg D_2=n$, set
\[
 (D_1,D_2)=D_1D_2(\vol_P)\in\ZZ.
\]

\begin{theorem}[Khovanskii--Pukhlikov, \cite{KhovanskiiPukhlikov93}, also cf.~{\cite[Theorem~5.1]{Kaveh11}}] Let $X$ be a smooth toric variety, $P$ the corresponding lattice polytope. Then
\[
 R_P\cong H^*(X,\ZZ)
\]
as graded rings: $(R_P)_k\cong H^{2k}(X,\ZZ)$. The pairing on $R_P$ corresponds to the Poincar\'e pairing on $H^*(X,\ZZ)$. 
\end{theorem}

If $P$ is simple, the elements of $R_P$ have a nice interpretation: they are algebraic combinations of linear differential  operators $\partial/\partial h_i$, where $h_i$ is a support number corresponding to a facet $F_i$ of $P$. Likewise, a monomial $\partial^k/\partial h_{i_1}\dots\partial h_{i_k}$ of degree corresponds to the face $F_{i_1}\cap\dots \cap F_{i_k}$ of codimension $k$ if this intersection is nonempty; otherwise it annihilates the volume polynomial and thus equals 0 in $R_P$. This establishes a correspondence between this description of $H^*(X,\ZZ)$ and the description given in \cite{Danilov78} or \cite[Chapter~12]{CoxLittleSchenck11}

\begin{remark} Sometimes it is more convenient to take the quotient of the space $V_P$ by translations: two polytopes are called equivalent if they can be obtained one from another by a translation. Denote the quotient space by $\overline{V_P}$. Since the volume is translation-invariant, $\vol_P$ defines a polynomial $\overline{\vol_P}$ of the degree $n$ on $\overline{V_P}$. Obviously, $\Diff_\ZZ(V_P)/\Ann(\vol_P)\cong\Diff_\ZZ(\overline{V_P})/\Ann(\overline{\vol_P})$.
\end{remark}

\begin{example} Let $P$ be a unit square. Then $S_P$ is formed by all rectangles with the sides parallel to the coordinate axes. There are natural coordinates on $\overline{V_P}$: the height and the width of a rectangle; denote them by $x$ and $y$. The volume polynomial is equal to $xy$, and $\Ann\vol_P=(\partial^2/\partial x^2,\partial^2/\partial y^2)$. Then
\[
 R_P=\langle 1,\partial/\partial x,\partial/\partial y, \partial^2/\partial x\partial y\rangle.
\]
This is nothing but the cohomology ring of $\PP^1\times \PP^1$.
\end{example}

\begin{remark} The notion of Khovanskii--Pukhlikov ring $R_P$ still makes sense for an arbitrary polytope $P$; it needs not to be simple. This will be our key observation in the next section, where we will consider the Khovanskii--Pukhlikov ring of a  Gelfand--Zetlin polytope, which is highly nonsimple.
However, for a nonsimple $P$ there is no such relation between the ring $R_P$ and the cohomology ring of the corresponding (singular) toric variety.
\end{remark}

\section{An approach to Schubert calculus via Khovanskii--Pukhlikov rings}\label{sec:gz}

In the last section we discuss a new approach to Schubert calculus on full flag varieties. It is based on the construction of Khovanskii--Pukhlikov ring, discussed in the previous section. We will mostly follow the paper \cite{KiritchenkoSmirnovTimorin12}.

\subsection{Gelfand--Zetlin polytopes}\label{ssec:gzpoly} Take a strictly increasing sequence of integers $\lambda=(\lambda_1<\lambda_2<\dots<\lambda_n)$. Consider a triangular tableau of the following form (it is called a \emph{Gelfand--Zetlin tableau}):
\begin{equation}\label{eq:GZ}
\begin{array}{ccccccccc}
 \lambda_1 && \lambda_2 &&\lambda_3 && \dots && \lambda_n\\
 & x_{11} && x_{12} && \dots && x_{1,n-1}\\
 && x_{21}  && \dots && x_{2,n-2}\\
&&& \ddots & \vdots\\
&&&&x_{n-1,1}
\end{array}
\end{equation}
We will interpret $x_{ij}$, where $i+j\leq n$, as coordinates in $\RR^{n(n-1)/2}$.  This tableau can be viewed a set of inequalities on the coordinates in the following way: for each triangle $\begin{array}{ccc} a && b\\& c\end{array}$ in this tableau, impose the inequalities $a\leq c\leq b$. This system of inequalities defines a bounded polytope in $\RR^{n(n-1)/2}$; it is not contained in any hyperplane. This polytope is called a \emph{Gelfand--Zetlin polytope}; we will denote it by $GZ(\lambda)$.

\begin{example} Here is our fundamental example: if $n=3$, the polytope $GZ(\lambda)$ is a polyhedron in $\RR^3$, presented on Figure~\ref{fig:GZpolytope}. The corresponding Gelfand--Zetlin tableau is as follows:
\[
 \begin{array}{ccccc}
  \lambda_1 && \lambda_2 &&\lambda_3\\
  & x && y\\
  && z
 \end{array}
\]
\end{example}

\begin{figure}[h!]
 \includegraphics{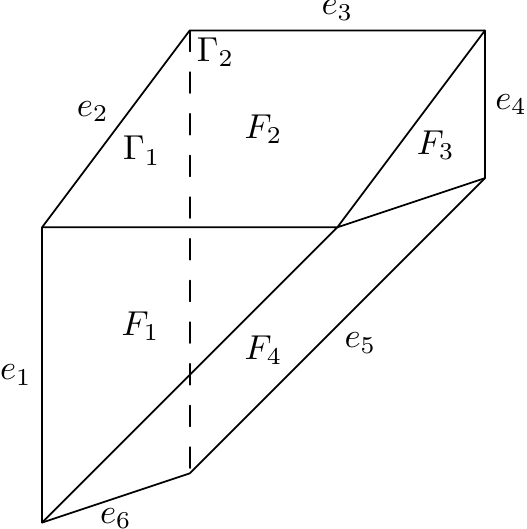}
 \caption{Gelfand--Zetlin polytope in dimension 3}\label{fig:GZpolytope}
\end{figure}

\begin{prop} For a given $n$, all Gelfand--Zetlin polytopes are analogous. The volume polynomial of $GZ(\lambda)$ is proportional to the Vandermonde determinant:
\[
 \vol_{GZ(\lambda)}=c\cdot\prod_{i>j}(\lambda_i-\lambda_j).
\]
\end{prop}

\begin{proof} The first part of the proposition is immediate. The second part follows from the fact that $\vol_{GZ(\lambda)}$ is a polynomial of degree $n(n-1)/2$ in $\lambda_1,\dots,\lambda_n$ that vanishes for $\lambda_i=\lambda_j$. Such a polynomial is unique up to a constant.
\end{proof}

Thus, the annihilator ideal of $\vol_{GZ(\lambda)}$ in $\Diff V_{GZ(\lambda)}=\ZZ[\partial/\partial\lambda_{1},\dots,\partial/\partial\lambda_{n}]$ equals the ideal generated by the symmetric polynomials in $\partial/\partial\lambda_{i}$ without the constant term. So we get the following corollary, probably first observed by Kiumars Kaveh. Essentially this is nothing but the Borel presentation for $H^*(\Fl(n))$, which we saw in Theorem~\ref{thm:borel}

\begin{corollary}[{\cite[Corollary~5.3]{Kaveh11}}]  The Khovanskii--Pukhlikov ring $R_{GZ}$ of the Gelfand--Zetlin polytope $GZ(\lambda)\subset\RR^{n(n-1)/2}$ is isomorphic to the cohomology ring of a complete flag variety $\Fl(n)$. An isomorphism is constructed as follows: $\partial/\partial\lambda_i$ is mapped to $-c_1(\ccL_i)$, where $c_1(\ccL_i)$ is the first Chern class of the $i$-th tautological line bundle $\ccL_i$ on $\Fl(n)$.
\end{corollary}

\subsection{Representation theory of $\GL(n)$ and Gelfand--Zetlin tableaux}\label{ssec:GZ_rep} Gelfand--Zetlin polytopes were introduced by I.~M.~Gelfand and M.~L.~Zetlin (sometimes also spelled Cetlin or Tsetlin) in 1950 (cf.~\cite{GelfandZetlin50}). The integer points in $GZ(\lambda)$ index a special basis, called the Gelfand--Zetlin basis, in the irreducible representation $V(\lambda)$ with the highest weight $\lambda$ of the group $\GL(n)$. Let us briefly recall some statements about the representation theory of $\GL(n)$ and the construction by Gelfand and Zetlin.

Let $(\CC^*)^n\cong T\subset \GL(n)$ be the subgroup of nondegenerate diagonal matrices, and let $V$ be a representation of $\GL(n)$. We say that $v\in V$ is a \emph{weight vector} if it is a common eigenvector for all diagonal matrices. This means that
\[
 (t_1,\dots,t_n)(v)=t_1^{\lambda_1}\dots t_n^{\lambda_n} v
\]
for some $\lambda=(\lambda_1,\dots,\lambda_n)\in\ZZ^n$. This set of integers is called the \emph{weight} of $v$.

We shall say that $\lambda$ is \emph{dominant} (or, respectively, \emph{antidominant}) if $\lambda_1\geq\dots\geq \lambda_n$ (resp. $\lambda_1\leq\dots\leq \lambda_n$), and \emph{strictly (anti)dominant} if all these inequalities are strict. 

We introduce a partial ordering on the set of weights, saying that $\lambda\preceq\mu$ if $\lambda_1+\dots+\lambda_i\geq \mu_1+\dots+\mu_i$ for each $1\leq i\leq n$. Moreover, a weight vector $v$ is said to be the \emph{highest} (resp.~\emph{lowest}) \emph{weight vector} if it is an eigenvector for the upper-triangular subgroup $B\subset\GL(n)$ (resp. $B^-\subset \GL(n)$):
\[
 b(v)=\lambda(b) v\qquad\text{for each}\quad b\in B.
\]

We say that a $\GL(n)$-module $V(\lambda)$ is a highest-weight (resp. lowest-weight) module with the highest (resp. lowest) weight $\lambda$ if a highest (resp. lowest) weight vector $v\in V(\lambda)$ is unique up to a scalar and has weight $\lambda$. In this case $V(\lambda)$ is spanned by the set of vectors $B^-(v)$ and $B(v)$, respectively. It is not hard to see that in this case $\lambda$ is indeed the highest (resp. lowest) weight in the sense of the partial ordering introduced earlier: for any weight $\mu$ of the module $V(\lambda)$ we have $\mu\preceq \lambda$ (or $\mu\succeq\lambda$, respectively).

The following theorem describes all irreducible rational finite-dimensional representations of $\GL(n)$. It can be found in any textbook on representation theory of Lie groups, such as \cite{FultonHarris91} or \cite{OnishchikVinberg90}. This theorem is usually formulated in terms of highest weights, but we prefer to give its equivalent form involving lowest weights instead.

\begin{theorem} For each antidominant weight $\lambda$ there exists an rational irreducible finite-dimensional $\GL(n)$-module $V(\lambda)$ with the lowest weight $\lambda$.  It is unique up to an isomorphism. Each rational irreducible finite-dimensional $\GL(n)$-module is isomorphic to some $V(\lambda)$.
\end{theorem}

One can also describe the set of all weights of a representation $V(\lambda)$:

\begin{prop} \begin{enumerate}
\item Each weight $\mu$ of $V(\lambda)$ is obtained from the lowest weight by adding a nonnegative integer combination of \emph{simple roots} $\alpha_i=(0,\dots,0,1,-1,0,\dots,0)$, where $1$ is on the $i$-th position, and $1\leq i\leq n-1$:
\[
\mu=\lambda+c_1\alpha_1+\dots+c_{n-1}\alpha_{n-1}, \qquad c_i\in\ZZ_+.
\]
In particular, the sum $\mu_1+\dots+\mu_n$ is equal for all weight vectors occuring in $V(\lambda)$ 

\item The set of weights is symmetric with respect to the action of the symmetric group $S_n$: if $\mu=(\mu_1,\dots,\mu_n)$ is a weight of $V(\lambda)$, then $\sigma(\mu):=(\mu_{\sigma(1)},\dots,\mu_{\sigma(n)})$ is again a weight of $V(\lambda)$. Moreover, the dimensions of their weight spaces are equal.
\end{enumerate}
\end{prop}

Of course, this proposition can be formulated in much greater generality for an arbitrary reductive group instead of $\GL(V)$, with its Weyl group action replacing the action of $S_n$ etc., but we will not need it here. An interested reader will find more details in \cite{FultonHarris91} or any other book on representations of Lie groups or algebraic groups.

Thus, the set of all weights of an irreducible representation $V(\lambda)$ is a finite set in $\ZZ^{n}$. It is contained in the hyperplane $x_1+\dots+x_n=\lambda_1+\dots+\lambda_n$. Its convex hull in $\RR^n=\ZZ^n\otimes \RR$ will be called the \emph{weight polytope} corresponding to $\lambda$ and denoted by $\wt(\lambda)$. It is a convex polytope of dimension $n-1$, symmetric under the standard action of $S_n$

\begin{xca} Show that if $\lambda=(\lambda_1,\lambda_2,\lambda_3)$ is a strictly antidominant weight, the corresponding weight polytope is a hexagon. Find the conditions for this hexagon to be regular. What happens for an antidominant, but not strictly antidominant $\lambda$?
\end{xca}

Gelfand--Zetlin polytopes appear in representation theory in the following way. Consider an irreducible representation $V(\lambda)$ of $\GL(n)$ with the lowest weight $\lambda=(\lambda_1\leq\lambda_2\leq\dots\leq\lambda_n)$ (not necessarily strictly antidominant). Inside $\GL(n)$ we can consider a subgroup stabilizing the subspace spanned by all basis vectors except the last one and the last basis vector; it consists of block-diagonal matrices with a block of size $n-1$ and the identity element in the bottom-right corner. Clearly, it is isomorphic to $\GL(n-1)$.

We can restrict our representation $V(\lambda)$ from $\GL(n)$ to $\GL(n-1)$, i.e., consider $V(\lambda)$ as a representation of the smaller group $\GL(n-1)$. This representation may become reducible; its irreducible components are indexed by their lowest weights $\lambda'=(\lambda_1'\leq\dots\leq\lambda_{n-1}')$:
\[
 \mathrm{Res}_{\GL(n)}^{\GL(n-1)}V(\lambda)=\bigoplus_{\lambda'} V({\lambda'})
\]

A key observation by Gelfand and Zetlin, made in \cite{GelfandZetlin50}, is that this representation of $\GL(n-1)$ is \emph{multiplicity-free}, i.e., all its irreducible components are non-isomorphic. Moreover, for each $\lambda'$ appearing in the decomposition, the following inequalities on the lowest weights $\lambda$ and $\lambda'$ hold:
\begin{equation}\label{eq:GZ_condition}
 \lambda_1\leq\lambda_1'\leq\lambda_2\leq\lambda_2'\leq\dots\leq\lambda_{n-1}'\leq \lambda_n.
\end{equation}

Now let us continue this procedure, restricting each of representations $V({\lambda'})$ to $\GL(n-2)$, and so on, until we reach $\GL(1)=\CC^*$. Each representation of $\CC^*$ is just a one-dimensional space. This means that we obtain a decomposition of $V(\lambda)$ into the direct sum of one-dimensional subspaces, which is defined by the chain of decreasing subgroups $\GL(n)\supset\GL(n-1)\supset\dots\supset\GL(1)$. Picking a vector on each of these lines, we obtain a \emph{Gelfand--Zetlin basis.} The elements of this basis are indexed by sequences of lowest weights of the groups in this chain: $\lambda,\lambda',\lambda'',\dots,\lambda^{(n)}$, such that any two neighboring weights in this sequence satisfy the inequalities~\ref{eq:GZ_condition}. So they are indexed exactly by Gelfand--Zetlin tableaux of type~\ref{eq:GZ}, consisting of integers. One can show that for each starting lowest weight $\lambda$, all possible integer Gelfald--Zetlin tableaux occur, so the Gelfand--Zetlin basis is indexed by the integer points inside the Gelfand--Zetlin polytope $GL(\lambda)$.

We can also consider the projection map that sends each row of a Gelfand--Zetlin tableau into the sum of its elements minus the sum of elements in the previous row, starting with the lowest row:
\[
\pi\colon \RR^{n(n-1)/2}\to \RR^n,\qquad
\begin{pmatrix}
x_{11}\\ \dots\\ x_{1,n-1}\\ \dots\\ x_{n-1,1}
\end{pmatrix}
\mapsto
\begin{pmatrix}
x_{n-1,1}\\x_{n-2,1}+x_{n-2,2}-x_{n-1,1}\\
\dots\\
x_{11}+\dots +x_{1,n-1}-x_{21}-\dots-x_{2,n-2}\\
\lambda_1+\dots+\lambda_n-x_{11}+\dots +x_{1,n-1}\end{pmatrix}.
\]
This map brings $GZ(\lambda)$ into the weight polytope $\wt(\lambda)$ of the representation $V_{\lambda}$.

\subsection{Faces of Gelfand--Zetlin polytopes}

We would like to follow the analogy with the toric case and treat the elements of the Khovanskii--Pukhlikov ring $R_{GZ}$ as linear combinations of faces of the polytope $GZ(\lambda)$. As we have seen before, this polytope is not integrally simple (even not simple). However, it can be \emph{resolved}: we can construct a simple polytope $\widehat{GZ(\lambda)}$ such that $GZ(\lambda)$ is obtained from it by contraction of some faces of codimension greater than one. In particular, this means that there is a natural bijection between the sets of facets of these two polytopes. This allows us to treat the elements of $R_{GZ}=R_{GZ(\lambda)}$ as elements of the bigger ring $\widehat{R_{GZ}}$ of the simple polytope $\widehat{GZ(\lambda)}$. The details of this construction can be found in~\cite[Section~2]{KiritchenkoSmirnovTimorin12} (in particular, see Subsection~2.4, where we treat in detail the example of a three-dimensional Gelfand--Zetlin polytope).

Let us describe the set of faces of the Gelfand--Zetlin polytope and the relations among them in $R_{GZ}$ and in $\widehat{R_{GZ}}$. The polytope is defined by a set of inequalities, represented by the diagram~\ref{eq:GZ}. Each face is obtained by turning some of these inequalities into equalities. In particular, each facet is defined by a unique equation: $x_{ij}=x_{i-1,j}$ or $x_{ij}=x_{i-1,j+1}$ for some pair $(i,j)$, where $i+j\leq n$. (We suppose that $x_{0,k}=\lambda_k$). Denote the facets of the first type by $F_{ij}$, and the facets of the second type by $F^-_{ij}$.

By differentiating the volume polynomial we can obtain all linear relations on facets:

\begin{prop}[{\cite[Proposition~3.2]{KiritchenkoSmirnovTimorin12}}]\label{prop:relations} 
The following linear relations hold in $\widehat{R_{GZ}}$:
\begin{equation}\label{eq:general}
 F_{ij}+F^-_{i+1,j-1}=F_{i,j}^-+F_{i+1,j}.
\end{equation}

 Moreover, all linear relations in $\widehat{R_{GZ}}$ are generated by these.
\end{prop}

We will represent faces of the Gelfand--Zetlin polytope symbolically by diagrams obtained from Gelfand--Zetlin tableaux by replacing all $\lambda_i$'s and $x_{ij}$'s  by dots, where each equality of type $x_{ij}=x_{i+1,j-1}$ or $x_{ij}=x_{i+1,j}$ is represented by an edge joining these dots. 


\begin{example} Consider again the Gelfand--Zetlin polytope in dimension 3. Denote its facets by $\Gamma_1$, $\Gamma_2$, $F_1$, $F_2$, $F_3$, $F_4$, as shown on Figure~\ref{fig:GZpolytope} ($\Gamma_1$ and $\Gamma_2$ are the two ``invisible'' trapezoid facets). Then the diagrams corresponding to these facets are shown on Figure~\ref{fig:facets}.
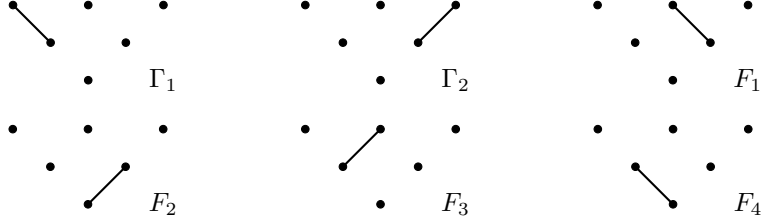
\begin{figure}[h!]
\begin{tikzpicture}  
 \draw[fill] (0,1) circle [radius=0.05];
 \draw[fill] (1,1) circle [radius=0.05];
 \draw[fill] (2,1) circle [radius=0.05];
 \draw[fill] (0.5,0.5) circle [radius=0.05];
 \draw[fill] (1.5,0.5) circle [radius=0.05];
 \draw[fill] (1,0) circle [radius=0.05];
 \draw[thick] (0,1)--(0.5,0.5);
\node at (2,0) {$\Gamma_1$};
\end{tikzpicture} 
\qquad\qquad
\begin{tikzpicture}   
 \draw[fill] (0,1) circle [radius=0.05];
 \draw[fill] (1,1) circle [radius=0.05];
 \draw[fill] (2,1) circle [radius=0.05];
 \draw[fill] (0.5,0.5) circle [radius=0.05];
 \draw[fill] (1.5,0.5) circle [radius=0.05];
 \draw[fill] (1,0) circle [radius=0.05];
 \draw[thick] (2,1)--(1.5,0.5);
\node at (2,0) {$\Gamma_2$};
\end{tikzpicture} 
\qquad \qquad
\begin{tikzpicture}  
 \draw[fill] (0,1) circle [radius=0.05];
 \draw[fill] (1,1) circle [radius=0.05];
 \draw[fill] (2,1) circle [radius=0.05];
 \draw[fill] (0.5,0.5) circle [radius=0.05];
 \draw[fill] (1.5,0.5) circle [radius=0.05];
 \draw[fill] (1,0) circle [radius=0.05];
 \draw[thick] (1,1)--(1.5,0.5);
 \node at (2,0) {$F_1$};
\end{tikzpicture} 

\bigskip
\begin{tikzpicture}  
 \draw[fill] (0,1) circle [radius=0.05];
 \draw[fill] (1,1) circle [radius=0.05];
 \draw[fill] (2,1) circle [radius=0.05];
 \draw[fill] (0.5,0.5) circle [radius=0.05];
 \draw[fill] (1.5,0.5) circle [radius=0.05];
 \draw[fill] (1,0) circle [radius=0.05];
 \draw[thick] (1,0)--(1.5,0.5);
 \node at (2,0) {$F_2$};
\end{tikzpicture} 
\qquad \qquad
\begin{tikzpicture}

 \draw[fill] (0,1) circle [radius=0.05];
 \draw[fill] (1,1) circle [radius=0.05];
 \draw[fill] (2,1) circle [radius=0.05];
 \draw[fill] (0.5,0.5) circle [radius=0.05];
 \draw[fill] (1.5,0.5) circle [radius=0.05];
 \draw[fill] (1,0) circle [radius=0.05];
 \draw[thick] (1,1)--(0.5,0.5);
 \node at (2,0) {$F_3$};
\end{tikzpicture} 
\qquad\qquad
\begin{tikzpicture}  
 
 \draw[fill] (0,1) circle [radius=0.05];
 \draw[fill] (1,1) circle [radius=0.05];
 \draw[fill] (2,1) circle [radius=0.05];
 \draw[fill] (0.5,0.5) circle [radius=0.05];
 \draw[fill] (1.5,0.5) circle [radius=0.05];
 \draw[fill] (1,0) circle [radius=0.05];
 \draw[thick] (1,0)--(0.5,0.5);
\node at (2,0) {$F_4$};
\end{tikzpicture} 
\caption{Face diagrams of facets of $GZ(\lambda)\subset\RR^3$}\label{fig:facets}
\end{figure}

From Proposition~\ref{prop:relations} we conclude that there are three independent linear relations on these faces:

\begin{eqnarray}\label{eq:relations}
 [\Gamma_1] &=& [F_3] + [F_4]\nonumber\\ \relax
 [\Gamma_2]  &=& [F_2] + [F_1]\\ \relax
 [F_2] &=& [F_4]\nonumber
\end{eqnarray}
\end{example}

\begin{remark}
 These linear relations also imply some nonlinear ones. For instance, we can take the four  face diagrams in the four-term relation from Proposition~\ref{prop:relations} and impose the same set of additional equalities on each of them; this would give a nonlinear four-term relation.
\end{remark}

\begin{xca} Show that for $GZ(\lambda)\subset\RR^3$, there are the following equalities on edges:
\[
 e_1=e_3=e_5\qquad\text{and} \qquad e_2=e_4=e_6
\]
(see Figure~\ref{fig:GZpolytope}).
\end{xca}

\subsection{Representing Schubert varieties by linear combinations of faces}

We have seen that elements of the cohomology ring of a full flag variety can be viewed as elements of $\widehat{R_{GZ}}$, i.e., as linear combinations of faces of the corresponding Gelfand--Zetlin polytope modulo the relations described in the previous subsection. Our next goal is to find a presentation for a given Schubert class $\sigma_w$ in $\widehat{R_{GZ}}\supset R_{GZ}$. This construction resembles the construction of pipe dreams. 

We will present $\sigma_w$ as a linear combination of faces of a certain special form, the so-called \emph{Kogan faces}. They were introduced in the Ph.D. thesis of Mikhail Kogan~\cite{Kogan00}.

\begin{definition} A face $F$ of $GZ(\lambda)$ is called a \emph{Kogan face} if it is obtained as the intersection of facets $F_{ij}$ for some $i,j$. Equivalently, $F$ is a Kogan face is it contains the vertex defined by the equations 
\begin{eqnarray*}
\lambda_1&=&x_{11}=x_{21}=\dots=x_{n-1,1},\\
\lambda_2&=&x_{12}=x_{22}=\dots=x_{n-2,2},\\
&&\dots\\
\lambda_{n-1}&=&x_{1,n-1}.
\end{eqnarray*}
\end{definition}

Now let us return to face diagrams from the previous subsection. Let $F$ be a Kogan face; all edges in its diagram go from northwest to southeast. We mark the edge going from $x_{i-1,j}$ to $x_{i,j}$ by a simple transposition $s_{i+j-1}\in S_n$ (recall that $1\leq i,j$ and $i+j\leq n$), as shown on Figure~\ref{fig:koganface}.

\begin{figure}[h!]
 \begin{tikzpicture}  
 \draw[fill] (-1,3) circle [radius=0.05];
 \draw[fill] (1,3) circle [radius=0.05];
 \draw[fill] (3,3) circle [radius=0.05];
 \draw[fill] (5,3) circle [radius=0.05];
 \draw[fill] (0,2) circle [radius=0.05];
 \draw[fill] (2,2) circle [radius=0.05];
 \draw[fill] (4,2) circle [radius=0.05];
 \draw[fill] (1,1) circle [radius=0.05];
 \draw[fill] (3,1) circle [radius=0.05];
 \draw[fill] (2,0) circle [radius=0.05];
 \draw[thick] (2,0)--(-1,3);
 \draw[thick] (3,3)--(4,2);
 \node[below] at (-0.5,2.5) {$s_1$};
 \node at (1.5,2.5) {$s_2$};
 \node[below] at (3.5,2.5) {$s_3$};
 \node[below] at (0.5,1.5) {$s_2$};
 \node at (2.5,1.5) {$s_3$};
 \node[below] at (1.5,0.5) {$s_3$};
\end{tikzpicture} 
\caption{The diagram of a Kogan face}\label{fig:koganface}
\end{figure}
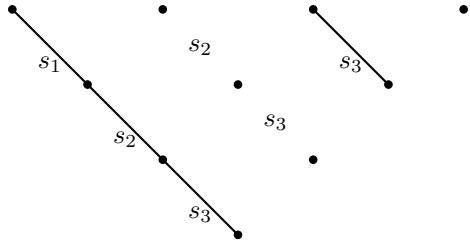

Now we take the word in $s_1,\dots,s_{n-1}$ obtained by reading the letters on the edges \emph{from bottom to top from left to right}. Thus, the diagram on Figure~\ref{fig:koganface} will produce the word $\underline{w}(F)=(s_3,s_2,s_1,s_3)$.

\begin{definition} Let $F$ be a Kogan face of codimension $k$, and let $\underline{w}(F)=(s_{i_1},\dots,s_{i_k})$ be the corresponding word. $F$ is said to be \emph{reduced} if the word $\underline{w}(F)$ is reduced, i.e., if $\ell(s_{i_1}\dots s_{i_k})=k$. In this case we will say that $F$ \emph{corresponds to the permutation} $w(F)=s_{i_1}\dots s_{i_k}$.
\end{definition}

\begin{example}
 The face shown on Figure~\ref{fig:koganface} is reduced; it corresponds to the permutation $s_3s_2s_1s_3=(4231)$.
\end{example}

\begin{example} Let $F$ be defined by equations $x_{12}=\lambda_2$, $x_{11}=x_{21}$. Then the corresponding word equals $\underline{w}(F)=(s_2,s_2)$, and $F$ is not reduced.
\end{example}

\begin{xca} Describe a natural bijection between reduced Kogan faces corresponding to $w\in S_n$ and pipe dreams with the same permutation.
\end{xca}

The following theorem is a direct analogue of the Fomin--Kirillov theorem (Theorem~\ref{thm:fominkirillov}). It shows that
each Schubert cycle can be represented by a sum of faces in exactly the same way as 
the corresponding Schubert polynomial can be represented by a sum of monomials.

\begin{theorem}[{\cite[Theorem~4.3]{KiritchenkoSmirnovTimorin12}}]\label{thm:main} A Schubert cycle $\sigma_w$, regarded as an element of the Gelfand--Zetlin polytope ring, can be represented by the sum of all reduced Kogan faces corresponding to the permutation $w$:
\[
 \sigma_w=\sum_{w(F_i)=w} [F_i]\in\widehat{R_{GZ}}.
\]
\end{theorem}

\begin{remark}
Despite the similarity between this theorem and the Fomin--Kirillov theorem, the
former cannot be formally deduced from the latter, since there is no term-by-term 
equality between monomials in the Schubert polynomial $\fS_w$ (which always lie in
the ring $R_{GZ}$) and the faces corresponding to the permutation $w$, which do not necessarily belong to $R_{GZ}$.
\end{remark}

\begin{remark} This correspondence between Schubert cycles and combinations of faces can be described geometrically in the following way. Consider a full flag variety $\Fl(n)\hookrightarrow \PP V(\lambda)$ embedded into the projectivization of the irreducible representation of $\GL(n)$ with a strictly dominant highest weight $\lambda$. It admits a toric degeneration, constructed by N.~Gonciulea and V.~Lakshmibai in \cite{GonciuleaLakshmibai96}. The exceptional fiber of this degeneration is a singular toric variety $\Fl^0(n)$ corresponding to the Gelfand--Zetlin polytope $GZ(\lambda)$. The images of Schubert varieties under this degeneration are (possibly reducible) $T$-stable subvarieties of $\Fl^0(n)$. This gives us the same presentation as in Theorem~\ref{thm:main}: each of their irreducible components is a Kogan face of $GZ(\lambda)$. The details can be found in \cite{KoganMiller05}. 
\end{remark}

\begin{example} Let $w=s_k$. Then there are $k$ faces of codimension 1 corresponding to $w$, and the Schubert divisor $\sigma_{s_k}$ is represented as
\[
 \sigma_{s_k}=[F_{1,k}]+[F_{2,k-1}]+\dots+[F_{k,1}].
\]
\end{example}

\begin{example} For $n=3$, we have the following presentation of Schubert cycles by faces of the Gelfand--Zetlin polytope (we keep the notation from Figure~\ref{fig:GZpolytope}):
\begin{eqnarray}
 \sigma_{s_1}&=&[\Gamma_1];\nonumber\\ \relax
 \sigma_{s_2}&=&[F_1]+[F_4];\nonumber\\ \relax
 \sigma_{s_1s_2}&=&[e_1];\label{eq:present}\\ \relax
 \sigma_{s_2s_1}&=&[e_6];\nonumber\\ \relax
 \sigma_{s_1s_2s_1}&=&[pt]\nonumber.
\end{eqnarray}
(The longest permutation corresponds to the class of point).
\end{example}

This presentation allows us to compute products of Schubert varieties. To multiply two cycles, $\sigma_w$ and $\sigma_v$, we need to represent them by linear combinations of mutually transversal faces and intersect these sets of faces. Using the relations in $\widehat{R_{GZ}}$, we can represent the result as the sum of certain Kogan faces; this sum corresponds to the linear combination of Schubert cycles $\sum c_{wv}^u \sigma_u=\sigma_w\cdot\sigma_v$.

Let us show this procedure on examples for $n=3$.
\begin{example} To begin with, let us multiply $\sigma_{s_1}$ by $\sigma_{s_2}$. Using (\ref{eq:present}), we write
\begin{multline*}
 \sigma_{s_1}\cdot \sigma_{s_2}=[\Gamma_1]\cdot ([F_1]+[F_4])=[\Gamma_1\cap F_1]+[\Gamma_1\cap F_4]=
[e_1]+[e_6]=\sigma_{s_1s_2}+\sigma_{s_2s_1}.
\end{multline*}

Here is another example. Compute $\sigma_{s_1}^2$. Here the equalities~(\ref{eq:present}) are not enough, since $\Gamma_1$ is not transversal to itself. So we need to replace one of the factors $[\Gamma_1]$ by an equivalent transversal combination of faces, using the relations~(\ref{eq:relations}):
\[
 \sigma_{s_1}^2=[\Gamma_1]\cdot([F_3]+[F_4])=[\Gamma_1\cap F_3]+[\Gamma_1\cap F_4]=0+[e_6]=\sigma_{s_2s_1}.
\]
The product $[\Gamma_1]\cdot[F_3]$ is zero since the corresponding faces do not intersect.
\end{example}

It turns out that the product of any two Schubert cycles can be computed in such a way:

\begin{theorem} For any two permutations $w$ and $v$, there are presentations of the corresponding Schubert cycles
\[
 \sigma_w=\sum [F_i]\qquad\text{and}\qquad \sigma_v=\sum [F_j'],
\]
such that each face $F_i$ is transversal to each of the $[F_j']$.
\end{theorem}

However, it is unclear whether the sum $\sigma_w\cdot\sigma_v=\sum_{i,j} [F_i\cap F_j]$ can be replaced by a linear combination of Kogan faces \emph{in a positive way}, that is, by using the relations~\ref{eq:general} without any subtractions. A positive answer to this question would imply a combinatorial proof of the positivity of structure constants for $H^*(\Fl(n),\ZZ)$. Now this is known to be true only for $n\leq 4$; this was shown by I.~Kochulin~\cite{Kochulin13} by direct computation.

\subsection{Demazure modules} The presentation of Schubert cycles by combinations of faces of Gelfand--Zetlin polytopes keeps track of some geometric information on Schubert varieties. As one example, we will describe the method of computing the degree of Schubert varieties.

Let $\lambda=(\lambda_1,\dots,\lambda_n)$ be a strictly antidominant weight, i.e., $\lambda_1<\lambda_2<\dots<\lambda_n$. As we discussed in the previous subsection, there exists a unique representation of $\GL(n)$ with this lowest weight; denote it by $V(\lambda)$. 
Let $v_-\in V(\lambda)$ be the lowest weight vector; this means that the line $\CC\cdot v_-$ is stable under the action of the lower-triangular subgroup $B^-\subset\GL(n)$. Since $\lambda$ is strictly antidominant, the stabilizer of $\CC\cdot v_-$ equals $B^-$ (for a non-strictly dominant highest weight, it can be bigger than $B^-$), so the $\GL(n)$-orbit of $\CC\cdot v_-$ in $\PP V(\lambda)$ is isomorphic to the full flag variety $\GL(n)/B\cong\Fl(n)$. So we have constructed an embedding of $\Fl(n)$ into $\PP V(\lambda)$.

Let us take the Schubert decomposition of $\Fl(n)$ associated with the lower-triangular subgroup $B^-$: the corresponding Schubert cells $\Omega_w'$ are just the orbits of the left action of $B^-$ on $\Fl(n)$. It turns out  that they behave nicely under this embedding: they are cut out from $\Fl(n)\subset \PP V(\lambda)$ by  projective subspaces. To make a more precise statement, we need the following definition.

\begin{definition} Let $w\in S_n$ be a permutation. Consider the vector $w_0w\cdot v_-$ and take the minimal $B$-submodule of $V(\lambda)$ containing $w_0w\cdot v_-$. Such a $B^-$-submodule is called a \emph{Demazure module} and denoted by $D_w(\lambda)$.
\end{definition}

\begin{example} The ``extreme cases'' are as follows: if $w=id$, the Demazure module equals the whole $\GL(n)$-representation space: $D_{id}(\lambda)=V(\lambda)$. For $w=w_0$ the vector $w_0^2v_-=v_-$ is the lowest weight vector, so it is $B^-$-stable, and $D_{w_0}(\lambda)=\CC\cdot v_-$.
\end{example}

\begin{remark} Demazure modules can also be described in terms of sections of line bundles on Schubert varieties: $D_w(\lambda)$ is the dual space to the space of global sections $H^0(X_w',\ccL_\lambda|_{X_w'})$, where $\ccL_\lambda|_{X_w'}$ is the restriction to $X_w'$ of the tautological line bundle on $\PP V(\lambda)$.
\end{remark}

\begin{prop} Schubert varieties can be obtained as intersections of $\Fl(n)$ with the projectivizations of the corresponding Demazure modules:
\[
 X_w'=\Fl(n)\cap\PP D_w(\lambda)\subset \PP V(\lambda).
\]
\end{prop}

Each $D_w(\lambda)$ is a $B^-$-module and, consequently, a $T$-module (as usual, $T$ is the diagonal torus in $\GL(n)$). We can consider its \emph{character}:
\[
 \ch D_w(\lambda)=\sum \mult_{D_w(\lambda)}(\mu) e^\mu,
\]
where the sum is taken over all weights of $D_w(\lambda)$, and $\mult_{D_w(\lambda)}(\mu)$ stands for the multiplicity of weight $\mu$, i.e. the dimension of the subspace of weight $\mu$ in $D_w(\lambda)$.

The character formula for Demazure modules was given by Michel Demazure \cite{Demazure74b}; however, its proof contained a gap, pointed out by Victor Kac. A correct proof was given by H.~H.~Andersen \cite{Andersen85}. We propose a method of computing the characters of Demazure modules for strictly dominant weights using our presentation of Schubert cycles by combinations of faces of Gelfand--Zetlin polytopes.

\begin{definition}\label{def:latticechar} Let $M\subset GZ(\lambda)\cap\ZZ^{n(n-1)/2}$ be a subset of $GZ(\lambda)$ (in our examples $M$ will be equal to a union of faces). Recall that in Subsection~\ref{ssec:GZ_rep} we have described a projection map $\pi\colon GZ(\lambda)\to \wt(\lambda)$ into the weight polytope with the highest weight $\lambda$. Denote by the \emph{lattice character} of $M$ the following formal sum taken over all integer points in $M$:
\[
 \ch M=\sum_{x\in M\cap\ZZ^{n(n-1)/2}} e^{\pi(x)}.
\]
\end{definition}

For example, if $M=GZ(\lambda)$, then $\ch M=\ch V(\lambda)$ is the character of the representation $V(\lambda)$. This formula can be generalized for all Demazure modules:

\begin{theorem}[{\cite[Theorem~5.1]{KiritchenkoSmirnovTimorin12}}]\label{thm:demazure} Let $w\in S_n$ be a permutation, and let $F_1,\dots, F_m$ be the set of reduced Kogan faces of $GZ(\lambda)$ corresponding to $w$ as in Theorem~\ref{thm:main}. Then the character of $D_w{\lambda}$ is equal to the lattice character of the union of these faces.
\[
 \ch D_w(\lambda)=\ch(F_1\cup\dots\cup F_m).
\]
\end{theorem}

Evaluating these characters at $1$, we get a formula for the dimension of $D_w(\lambda)$.
\begin{corollary} With the same notation,
\[
 \dim D_w(\lambda)=\#((F_1\cup\dots\cup F_m)\cap\ZZ^{n(n-1)/2}).
\]
\end{corollary}

Recall that the \emph{degree} of a $d$-dimensional projective variety $X\subset\PP^n$ is defined as the number of points in the intersection of $X$ with a generic $(n-d)$-plane. Of course, the degree depends upon the embedding of $X$ into $\PP^n$. Theorem~\ref{thm:demazure} also provides a way to compute the degrees of Schubert varieties $X_w'$. It turns out to be equal to the total \emph{volume} of the faces corresponding to $w$ times a certain constant.

To be more precise, let $F\subset GZ(\lambda)$ be a $d$-dimensional face of $GZ(\lambda)$. Let us normalize the volume form on its affine span $\RR F$ in such a way that the covolume of the lattice $\ZZ^d\cap\RR F$ in $\RR F$ would be equal to 1. Then the following theorem holds.

\begin{theorem}[{\cite[Theorem~5.4]{KiritchenkoSmirnovTimorin12}}] Let $w\in S_n$. Then, with the notation of Theorem~\ref{thm:demazure}, the degree of the Schubert variety $X_w'\subset \PP V(\lambda)$ equals
\[
 \deg_\lambda X_w'=\left(\frac{n(n-1)}{2}-\ell(w)\right)!\cdot\sum_{i=1}^m \vol(F_i).
\]
\end{theorem}

\bibliographystyle{amsalpha}
\bibliography{flags}

\end{document}